\documentclass[a4paper,12pt]{article}

\usepackage{amsfonts}
\usepackage{amsmath}
\usepackage{amssymb}
\usepackage{chngcntr}
\usepackage[dvipsnames]{xcolor}
\usepackage{enumerate}
\usepackage{faktor}
\usepackage[T1]{fontenc}
\usepackage{graphicx}
\usepackage{hyperref}
\usepackage{mathrsfs}
\usepackage{mathtools}
\usepackage{pxfonts}
\usepackage{soul}
\usepackage{stmaryrd}

\SetSymbolFont{stmry}{bold}{U}{stmry}{m}{n}

\newtheorem{definition}{Definition}
\newtheorem{fact}[definition]{Fact}
\newtheorem{comment}[definition]{Comment}
\newtheorem{lemma}[definition]{Lemma}
\newtheorem{proposition}[definition]{Proposition}
\newtheorem{theorem}[definition]{Theorem}
\newtheorem{corollary}[definition]{Corollary}
\newtheorem{example}[definition]{Example}

\counterwithin{definition}{section}
\counterwithin{equation}{definition}

\newcounter{proofcounter}

\renewcommand{\theequation}{%
	\ifnum\theproofcounter=0
		\arabic{section}\Alph{equation}%
	\else
		\arabic{section}.\arabic{definition}.\arabic{equation}%
	\fi}

\newenvironment{proof}{\setcounter{proofcounter}{1}\noindent\textbf{Proof} }{\hfill$\dashv$\vspace{0.2cm}\setcounter{proofcounter}{0}}

\newcommand{\cut}[1]{}

\newcommand{\qedspace}[1]{\vspace{-#1}}

\newcommand{\defstyle}[1]{\textbf{\boldmath #1}}

\newcommand{\ef}[3]{\mathrm{EF}_{#1\!}\left({#2},{#3}\right)}

\newcommand{\lang}{\mathcal{L}}

\newcommand{\for}[1]{\mathsf{#1}}

\newcommand{\res}[2]{{#1}{\upharpoonright}_{#2}}
\newcommand{\tres}[2]{{\tree{#1}}\setminus{#2}}

\newcommand{\eq}[1]{\equiv^{#1}}

\newcommand{\cfor}[3]{\llbracket(\tree{#1};#2)\rrbracket^{#3}}

\newcommand{\tree}[1]{\mathfrak{#1}}
\newcommand{\ls}[1]{<_{\tree{#1}}}
\newcommand{\lseq}[1]{\leqslant_{\tree{#1}}}
\newcommand{\gr}[1]{>_{\tree{#1}}}

\newcommand{\comp}[1]{\smile_{\tree{#1}}}

\newcommand{\double}[1][]{\ll_{\tree{#1}}}

\newcommand{\block}[3][]{\mathcal{B}_{#2}^{\,{#1}}\!\left(#3\right)}
\newcommand{\isablock}[2]{\for{block}^{\,{#1}}_{\,{#2}}}
\newcommand{\spec}[3]{\mathrm{spec}^{\,{#2}}_{\,{#1}}\!\left({#3}\right)}

\newcommand{\psubset}[2]{{#1}({#2})}
\newcommand{\psubtree}[2]{\tree{#1}({#2})}

\newcommand{\grset}[2]{{#1}^{>{#2}}}
\newcommand{\grtree}[2]{\tree{#1}^{>{#2}}}
\newcommand{\greqset}[2]{{#1}^{\geqslant{#2}}}
\newcommand{\greqtree}[2]{\tree{#1}^{\geqslant{#2}}}
\newcommand{\compset}[2]{{#1}^{\not>{#2}}}
\newcommand{\comptree}[2]{\tree{#1}^{\not>{#2}}}
\newcommand{\compeqset}[2]{{#1}^{\not\geqslant{#2}}}
\newcommand{\compeqtree}[2]{\tree{#1}^{\not\geqslant{#2}}}

\newcommand{\cfset}[2]{{#1}_{\textsf{\tiny cf}}^{\,{#2}}}
\newcommand{\cftree}[2]{\tree{#1}_{\textsf{\tiny cf}}^{{#2}}}

\newcommand{\opensecset}[3]{{#1}_{\textsf{\tiny cs}}^{\,({#2},\,{#3})}}
\newcommand{\opensectree}[3]{\tree{#1}_{\textsf{\tiny cs}}^{\,({#2},\,{#3})}}

\newcommand{\rightclosedsecset}[3]{{#1}_{\textsf{\tiny cs}}^{\,({#2},\,{#3}]}}
\newcommand{\rightclosedsectree}[3]{\tree{#1}_{\textsf{\tiny cs}}^{\,({#2},\,{#3}]}}

\newcommand{\sumtree}[3]{\tree{#1} \! +_{#3} \! \tree{#2}}

\allowdisplaybreaks

\begin{document}

\title{Bounded elementary extensions of trees with unbounded paths}
\author{\textsc{Ruaan Kellerman} \\ Department of Mathematics and Applied Mathematics, \\
	University of Pretoria, Pretoria, South Africa \\
	Email: \texttt{ruaan.kellerman@up.ac.za}}

\maketitle

\abstract{A tree is a partially ordered set that is downwards linear and downwards connected.  A tree is called bounded when each of its paths (i.e.\ maximal linearly ordered subsets) contains a greatest element.  In a bounded tree, each path can be defined by a first-order formula using the leaf of the path as parameter.  Bounded trees can be used to model computational systems such as Zeno machines whereby the leaf of a path represents the state to which an infinitely long sequence of computations converges, or a state that is assigned to a computational sequence that loops.  We identify a sufficient condition under which certain trees that are not bounded, can be elementarily embedded in trees that are bounded.  Several tree operations are also given, and Feferman-Vaught style preservation properties for these operations are proved.}

\bigskip


\bigskip


\section{Introduction}

A tree is a structure $\tree{T} := (T;<)$ in which $T$ is a non-empty set, the elements of which are called \defstyle{nodes}, and $<$ is a binary relation on $T$ that is irreflexive, downwards linear (i.e.\ for each $t \in T$, the set $\{x \in T : x \leqslant t\}$ is totally ordered by $<$) and downwards connected (i.e.\ for all $s,t \in T$ there exists $x \in T$ such that $x \leqslant s$ and $x \leqslant t$).\footnote{Observe that the class of trees can be defined by a finite set of first-order sentences.}  A maximal linearly ordered subset of $T$ will be called a \defstyle{path}.\footnote{Paths are commonly also referred to as \emph{branches} in the literature.}  A node will be called a \defstyle{leaf} when it is maximal with respect to $<$; the set of leaves can be defined by the formula
$$\for{leaf}(x) := \neg \exists y (x < y).$$
A node that is not a leaf will be called an \defstyle{internal} node.  A path will be called \defstyle{bounded} when it has a greatest element, otherwise it will be called \defstyle{unbounded}.\footnote{From \cite{GurevichShelahChapter}.}  A tree will be called \defstyle{bounded} when each of its paths is bounded.  The class of bounded trees is of natural interest.  In a bounded tree, each path can be defined by a first-order formula using the leaf of the path as parameter, which in effect permits for a form of restricted monadic quantification over paths in such trees.  Bounded trees can also be used to model computational systems such as Zeno machines whereby the leaf of a path represents the state to which an infinitely long sequence of computations converges, or a state that is assigned to a computational sequence that loops.

Prominent classes of trees of which the first-order theories have been determined, include the class of finitely branching trees (see \cite{Schmerl}), the class of well-founded trees (see \cite{Doets}), the class of finite trees (see \cite{BackofenRogersVijay-Shankar}), and the classes of trees of which all paths are isomorphic to the natural numbers, the integers, the rational numbers, and the reals (see \cite{GorankoKellermanJSL}).  Despite the class of bounded trees being of natural interest, its first-order theory is not known however.

Let $A$ be a path in $\tree{T}$ that is definable by a first-order formula, possibly using parameters, i.e.\ there exists a (possibly empty) tuple of nodes $\bar{a}$ from $\tree{T}$, and a first-order formula $\varphi(x,\bar{y})$ (where $\bar{y}$ is a tuple of variables of the same arity as $\bar{a}$) such that $\left\{ t \in T : (\tree{T};\bar{a}) \models \varphi(t,\bar{a}) \right\} = A$.  If $A$ has a greatest node, this fact can be expressed in $(\tree{T};\bar{a})$ by the first-order sentence $\exists x \left( \varphi(x,\bar{a}) \wedge \for{leaf}(x) \right)$.  Let $\for{path}_{\varphi}$ be a sentence\footnote{
	$\for{path}_{\varphi}$ can be taken as the sentence $\exists \bar{y} \left( \rho_1 \wedge \rho_2 \wedge \rho_3 \wedge \rho_4 \right)$, where the $\rho_i$ are the following sentences:\smallskip
	
	\ \ \ $\rho_1$: $\exists x (\varphi(x,\bar{y}))$ (the set defined by $\varphi$ is non-empty);\smallskip
	
	\ \ \ $\rho_2$: $\forall x \forall z \left( \left( \varphi(x,\bar{y}) \wedge \varphi(z,\bar{y}) \right) \rightarrow \left( (x < z) \vee (x = z) \vee (z < x) \right) \right)$ (the set defined by $\varphi$ is totally ordered);\smallskip
	
	\ \ \ $\rho_3$: $\forall x \forall z \left( \left( \varphi(x,\bar{y}) \wedge (z < x) \right) \rightarrow \varphi(z,\bar{y}) \right)$ (the set defined by $\varphi$ is downwards closed);\smallskip
	
	\ \ \ $\rho_4$: $\neg \exists x \forall z \left( \varphi(z,\bar{y}) \rightarrow (z < x) \right)$ (no node is greater than every node in the set defined by $\varphi$).\smallskip}
that expresses the fact that $\varphi(x,\bar{y})$ defines a path in $\tree{T}$.  If every definable path in $\tree{T}$ contains a leaf, this can be expressed by the axiom scheme that consists of the sentences $\forall \bar{y} \left( \for{path}_{\psi} \rightarrow \exists x \left( \psi(x,\bar{y}) \wedge \for{leaf}(x) \right) \right)$ for all formulas $\psi(x,\bar{y})$.  
Paths in a tree need not be definable by a first-order formula however, which makes the problem of determining the first-order theory of the class of bounded trees non-trivial.

We will describe two examples of this.  In the first example, we have a tree $\tree{B}_0$ with infinitely many unbounded paths that can be made into an elementarily equivalent bounded tree $\tree{B}_{\omega+1}$ by simply augmenting each unbounded path with a leaf, while in the second example, we encounter a tree $\tree{T}$ that also has infinitely many unbounded paths, but cannot be converted into an elementarily equivalent bounded tree by augmenting each unbounded path with a leaf, even though $\tree{T}$ itself is a model of the first-order theory of the class of bounded trees.

\begin{example} \label{Ex:BinaryTree}
	Consider the tree $\tree{B}_{\omega+1}$, depicted in Figure \ref{Fig:BinaryTree}, of which each path is isomorphic to the ordinal $\omega+1$, with each internal node having exactly two immediate successors, and with each leaf having no siblings (i.e.\ if $t$ and $s$ are leaves with $\{x : x < t\} = \{x : x < s\}$, then $t = s$).
	
	Nodes in $\tree{B}_{\omega+1}$ can be represented as finite binary strings, and paths can be represented as binary strings of denumerable length.  The cardinality of the set of paths in $\tree{B}_{\omega+1}$ is therefore $2^{\aleph_0}$, hence the cardinality of the set of leaves in $\tree{B}_{\omega+1}$ is also $2^{\aleph_0}$.
	\begin{figure}
		\begin{center}
			\includegraphics[scale=0.33]{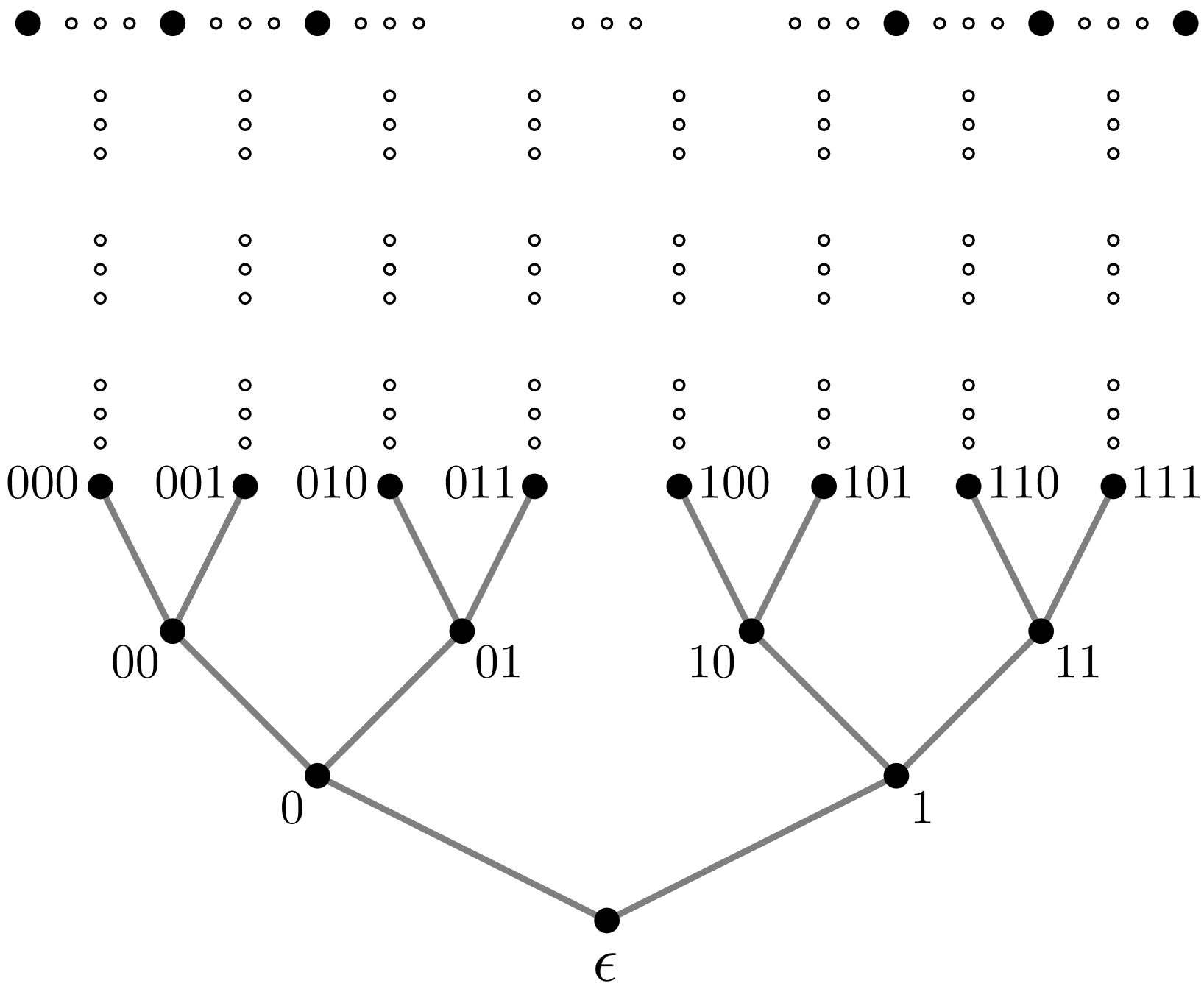}
			\caption{The tree $\tree{B}_{\omega+1}$ in Example \ref{Ex:BinaryTree}.}
			\label{Fig:BinaryTree}
		\end{center}
	\end{figure}
	
	By the Downards L\"owenheim-Skolem Theorem, $\tree{B}_{\omega+1}$ has a countable elementary substructure $\tree{B}_0$.  Each internal node in $\tree{B}_0$ must still have two immediate successors, and for each internal node $t$ in $\tree{B}_0$ there must exist a leaf $s$ in $\tree{B}_0$ for which $t < s$, so it follows that the cardinality of the set of leaves in $\tree{B}_0$ must be $\aleph_0$.  However, $\tree{B}_0$ still contains uncountably many paths.  $\tree{B}_0$ therefore contains two types of path: paths that are isomorphic to $\omega$, of which there are uncountably many, and paths that are isomorphic to $\omega+1$, of which there are countably infinitely many.  Remarkably, $\tree{B}_0$ has more unbounded paths that are isomorphic to $\omega$ than it has bounded paths that are isomorphic to $\omega+1$, even though $\tree{B}_0$ is elementarily equivalent to $\tree{B}_{\omega+1}$, in which all paths are bounded and isomorphic to $\omega+1$.
	
	Note that none of the paths in $\tree{B}_0$ that are isomorphic to $\omega$ are first-order definable, for suppose to the contrary that $A$ is a path in $\tree{B}_0$ that is isomorphic to $\omega$, and suppose that the formula $\varphi(x,\bar{y})$ defines $A$ in $\tree{B}_0$ with parameters $\bar{a}$.  Let $\psi$ be the sentence
	$$\for{path}_{\varphi(x,\bar{a})} \wedge \forall x \Big( \varphi(x,\bar{a}) \rightarrow \exists z \big( \varphi(z,\bar{a}) \wedge (x < z) \big) \Big).$$
	Then $(\tree{B}_0;\bar{a}) \models \psi$, and since $\tree{B}_0 \preceq \tree{B}_{\omega+1}$ then also $(\tree{B}_{\omega+1};\bar{a}) \models \psi$.  However, this contradicts the fact that every path in $\tree{B}_{\omega+1}$ is isomorphic to $\omega+1$.
\end{example}

\begin{example}
	In \cite[Example 4]{KellermanAML}, a tree $\tree{T}$ is constructed with the property that $\tree{T}$ is not bounded, even though $\tree{T}$ is a model of the first-order theory of the class of bounded trees.  Moreover, if $\widetilde{\tree{T}}$ is the tree that is obtained by augmenting any unbounded path in $\tree{T}$ with a leaf, then $\widetilde{\tree{T}} \not\equiv \tree{T}$.
\end{example}

This paper concerns itself with the problem of determining the first-order theory of the class of bounded trees.\footnote{This entails determining a consistent and decidable set of first-order sentences $\Sigma$ such that for each model $\tree{T}$ of $\Sigma$ and each natural number $n$, there exists a bounded tree $\tree{S}$ such that $\tree{T}$ and $\tree{S}$ satisfy the same first-order sentences of quantifier rank at most $n$.}  It is a continuation of the work done in \cite{KellermanAML}.  Specifically, we identify a sufficient condition under which a well-founded tree $\tree{T}$ of which each path is either a limit ordinal or an ordinal of the form $\alpha+1$, where $\alpha$ is a limit ordinal, and in which leaves have no siblings, can be elementarily embedded in a bounded tree $\tree{T}'$.  In fact, this bounded tree $\tree{T}'$ will be an end-extension of the original tree $\tree{T}$, obtained from $\tree{T}$ by simply augmenting unbounded paths with leaves.\footnote{$\tree{T}'$ is an \emph{end-extension} of $\tree{T}$ when $\tree{T}$ is a substructure of $\tree{T}'$ and, whenever $a$ is a node in $\tree{T}'$ and $b$ is a node in $\tree{T}$ with $a < b$, then $a$ is in $\tree{T}$ too.}  The assumptions of well-foundedness, of each path being either a limit ordinal or an ordinal of the form $\alpha+1$, and of leaves having no siblings, can be relaxed somewhat, but we use these stronger assumptions in the interest of brevity.\footnote{Well-foundedness is needed to ensure that a tree can be partitioned along a path into co-forests.  Well-foundedness can be replaced with the weaker assumption of \emph{weak branching completeness} (see \cite{GorankoKellermanZanardoII}), which allows for a tree to be partitioned along a path into so-called side forests.}

The paper is structured as follows.  In \emph{Section \ref{Sec:Preliminaries}: Preliminaries}, we review important definitions and results pertaining to trees and first-order logic that will be needed in the remainder of the paper.  In \emph{Section \ref{Sec:Operations}: Operations on trees}, we define some operations on trees that will be needed in subsequent constructions.  In \emph{Section \ref{Sec:Properties}: Properties of trees}, we define the sufficiency properties of focality and variegation that will be needed for our main result.  In \emph{Section \ref{Sec:Equivalence}: Equivalence of cross sections}, we obtain a result on the elementary equivalence of cross sections.  Finally, in \emph{Section \ref{Sec:Leaves}: Adding leaves to unbounded paths}, we obtain our main sufficiency result.

\section{Preliminaries} \label{Sec:Preliminaries}

\subsection{Trees}

The underlying set $T$ of a tree $\tree{T} = (T;\ls{T})$ will sometimes be indicated using the notation $|\tree{T}|$, and the order relation $\ls{T}$ will be written simply as $<$ provided there is no ambiguity.  Context permitting, it will be understood that $\tree{T}$ has underlying set $T$, $\tree{S}$ has underlying set $S$, etc.  $\tree{T}$ will be called \defstyle{trivial} when $T$ is a singleton set and $\ls{T}$ is the empty relation.

The language of $\tree{T}$ may sometimes be enriched by adding to it a tuple $\bar{a}$ of nodes from $T$ that are treated as constants, or a subset $A$ of $T$ that is treated as a unary relation, to obtain structures of the form $(\tree{T};\bar{a})$ and $(\tree{T};A)$ respectively.  At most finitely many constant symbols and unary relation symbols will be added to the language in this manner.  $\lang(\tree{T})$ will indicate the first-order language containing equality and the binary relation symbol $<$, along with all (if any) additional constant and unary relation symbols of $\tree{T}$.  By an $\lang$-tree, respectively $\lang$-formula, we mean a tree, respectively formula, in some language $\lang = \lang(\tree{S})$.

A subset $X$ of $T$ will be called a \defstyle{stem} when $(X;\res{<}{X})$ is a downwards convex (i.e.\ if $x \in X$ and $y < x$ then $y \in X$) linear order.  $X$ will be called a \defstyle{branch} when for some path $Y$, $X \subseteq Y$ and $Y \setminus X$ is a stem.

For $X,Y \subseteq T$, the notation $X \ll Y$ will be short for $x < y$ for all $x \in X$ and $y \in Y$; if e.g.\ $X = \{x\}$ then $X \ll Y$ will be written simply as $x \ll Y$, etc.  Intervals will carry their usual meanings, e.g.\ $(a,b] = \{ x \in T : a < x \leqslant b \}$ and $(-\infty,b) = \{ x \in T : x < b \}$.  For any subset $B$ of $T$, $\tres{T}{B}$ will denote the restriction of $\tree{T}$ to the set $T \setminus B$; if $B = \{b\}$ then $\tres{T}{B}$ will be written simply as $\tres{T}{b}$.  The comparability of two nodes can be expressed by the formula
$$x \smile y \ := \ (x < y) \vee (y < x) \vee (x = y).$$

\subsection{Logic}

Two structures will be called \defstyle{$n$-equivalent} when they satisfy the same first-order sentences of quantifier rank at most $n$.  $n$-equivalence of $\tree{A}$ and $\tree{B}$ will be denoted $\tree{A} \eq{n} \tree{B}$.  We will make use of the following ``back-and forth'' characterisation of $n$-equi\-va\-lence.

\begin{fact} \cite[Theorem 3.18 and Proposition 3.20]{DoetsBMT} \label{Thm:BackAndForth}
	Let $\tree{A}$ and $\tree{B}$ be structures of the same signature.  Then $\tree{A} \eq{n+1} \tree{B}$ if and only if both of the following conditions hold:
	\begin{enumerate}[(i)]
		\item
			for each $a \in A$ there exists $b \in B$ such that $(\tree{A};a) \eq{n} (\tree{B};b)$;
		\item
			for each $b \in B$ there exists $a \in A$ such that $(\tree{A};a) \eq{n} (\tree{B};b)$.
	\end{enumerate}
\end{fact}

We will use the terminology and notation of \cite[Section 1.6]{DoetsThesis} in relation to characteristic formulas.  The following fact is of relevance to the subsequent definition.

\begin{fact} \cite[Proposition 3.23]{DoetsBMT} \label{Thm:FiniteSpectrum}
	In a first-order language having at most finitely many constant symbols and relation symbols, and no function symbols, there are, for any natural numbers $n$ and $k$, only finitely many pairwise logically inequivalent formulas of quantifier rank $n$ having $k$ free variables.
\end{fact}

\begin{definition} \cite[Definition 1.6.1]{DoetsThesis} \label{Def:CharacteristicFormula}
	Let $\tree{A}$ be a structure in a first-order language having at most finitely many constant symbols and relation symbols, and no function symbols.  Let $n$ be a natural number and let $\bar{a}$ be a $k$-tuple of elements from $\tree{A}$.  There exists a formula $\cfor{A}{\bar{a}}{n}(\bar{x})$, called the \defstyle{$n$-characteristic formula} of $\tree{A}$ with respect to $\bar{a}$, having the following properties:
	\begin{enumerate}[(i)]
		\item
			$\cfor{T}{\bar{a}}{n}(\bar{x})$ is a formula with $k$ free variables $\bar{x}$ and quantifier rank $n$.
		\item
			For any structure $\tree{B}$ in the same language as $\tree{A}$, and for any $k$-tuple $\bar{b}$ of elements from $\tree{B}$,
			$$\tree{B} \models \cfor{A}{\bar{a}}{n}(\bar{b}) \ \Longleftrightarrow \ (\tree{A};\bar{a}) \eq{n} (\tree{B};\bar{b}) \ \Longleftrightarrow \ \cfor{A}{\bar{a}}{n} \equiv \cfor{B}{\bar{b}}{n}.$$
			In particular, $(\tree{A};\bar{a}) \models \cfor{A}{\bar{a}}{n}(\bar{a})$.
	\end{enumerate}
\end{definition}

For a given tree language $\lang$, the formulas $\tau^n_1(x),\ldots,\tau^n_m(x)$ will indicate all of the $n$-characteristic $\lang$-formulas with one free variable, up to logical equivalence.

The $n$-round Ehrenfeucht-Fra\"iss\'e game (see e.g.\ \cite{DoetsBMT}) on $\tree{A}$ and $\tree{B}$ will be abbreviated $\ef{n}{\tree{A}}{\tree{B}}$.

\section{Operations on trees}  \label{Sec:Operations}

\subsection{Sums of trees}

Let $\tree{T}$ be a tree and $\tree{F}$ be a forest, and let $L$ be a path in $\tree{T}$.  The \defstyle{sum} of $\tree{T}$ and $\tree{F}$ with respect to $L$ is the tree $\sumtree{T}{F}{L}$ defined as follows:
\begin{itemize}
	\item
		the underlying set of $\sumtree{T}{F}{L}$ is $(T \times \{0\}) \cup (F \times \{1\})$;
	\item
		the order relation of $\sumtree{T}{F}{L}$ is defined as
		\begin{multline} \nonumber
			{<_{\sumtree{T}{F}{L}}} := \big\{ \big((x,0),(y,0)\big) : x \ls{T} y \big\} \cup \big\{ \big((x,1),(y,1)\big) : x \ls{F} y \big\} \ \cup \\
			\big\{ \big((x,0),(y,1)\big) : x \in L \ \text{and} \ y \in F \big\};
		\end{multline}
	\item
		for each constant $a$ in $\tree{T}$ (respectively $\tree{F}$), the element $(a,0)$ (respectively $(a,1)$) occurs as a constant in $\sumtree{T}{F}{L}$;
	\item
		for each relation $A$ in $\tree{T}$ (respectively $\tree{F}$), the relation $A \times \{0\}$ (respectively $A \times \{1\}$) occurs as a relation in $\sumtree{T}{F}{L}$.
\end{itemize}
The tree $\sumtree{T}{F}{L}$ is depicted in Figure \ref{Fig:Sum}.  Given $t \in T$ or $s \in F$ we will, provided no ambiguity arises, simply refer to the node $(t,0) \in \left|\sumtree{T}{F}{L}\right|$ as $t$, and to the node $(s,1) \in \left|\sumtree{T}{F}{L}\right|$ as $s$.  Similarly, for $A \subseteq T$ or $B \subseteq F$, the relation $A \times \{0\} \subseteq \left|\sumtree{T}{F}{L}\right|$ will be referred to simply as $A$, and $B \times \{1\} \subseteq \left|\sumtree{T}{F}{L}\right|$ will be referred to simply as $B$.
\begin{figure}
	\begin{center}
		\includegraphics[scale=0.8]{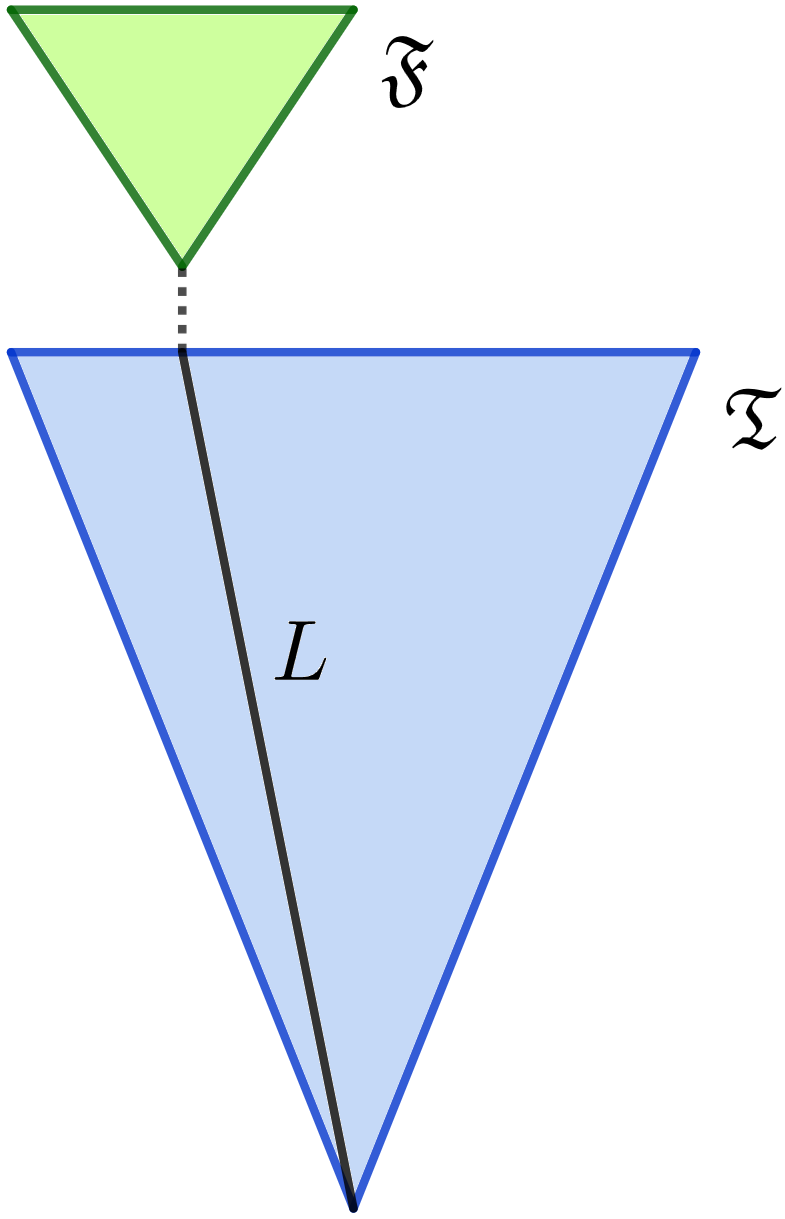}
		\caption{The sum $\sumtree{T}{F}{L}$.}
		\label{Fig:Sum}
	\end{center}
\end{figure}

\begin{lemma} \label{Thm:SumsOfTrees}
	Let $\tree{T}$ and $\tree{S}$ be $\lang_1$-trees and let $\tree{F}$ and $\tree{G}$ be $\lang_2$-forests.  Let $A$ and $B$ be paths in $\tree{T}$ and $\tree{S}$ respectively.  If $(\tree{T};A) \eq{n} (\tree{S};B)$ and $\tree{F} \eq{n} \tree{G}$ then $\left(\sumtree{T}{F}{A};A\right) \eq{n} \left(\sumtree{S}{G}{B};B\right)$.
\end{lemma}

\begin{proof}
	Consider the game $\ef{n}{\left(\sumtree{T}{F}{A};A\right)}{\left(\sumtree{S}{G}{B};B\right)}$.  Whenever Player I chooses a node from $T$ or $S$, Player II responds using her winning strategy for the game $\ef{n}{(\tree{T};A)}{(\tree{S};B)}$, and whenever Player I chooses a node from $F$ or $G$, Player II responds using her winning strategy for the game $\ef{n}{\tree{F}}{\tree{G}}$.  Let $t_k \in \left|\sumtree{T}{F}{A}\right|$ and $s_k \in \left|\sumtree{S}{G}{B}\right|$ be the nodes so chosen in the $k$-th round of the game.  Note that whenever $t_i \in T$, $t_j \in F$, $s_i \in S$ and $s_j \in G$,
	$$t_i <_{\sumtree{T}{F}{A}} t_j \ \Longleftrightarrow \ t_i \in A \ \Longleftrightarrow \ s_i \in B \ \Longleftrightarrow \ s_i <_{\sumtree{S}{G}{B}} s_j.$$
	It follows that $(\sumtree{T}{F}{A};A,t_1,\ldots,t_n)$ and $(\sumtree{S}{G}{B};B,s_1,\ldots,s_n)$ satisfy the same atomic sentences, as required.
\end{proof}

\subsection{Cones and anticones}

Let $\tree{T}$ be a tree and let $a$ be a node in $\tree{T}$.  Define the sets
\begin{eqnarray*}
	\grset{T}{a} & := & \{x \in T: a \ls{T} x\}, \\
	\greqset{T}{a} & := & \{x \in T: a \lseq{T} x\}, \\
	\compset{T}{a} & := & T \setminus \grset{T}{a} \ \ \ \text{and} \\
	\compeqset{T}{a} & := & T \setminus \greqset{T}{a}.
\end{eqnarray*}
The structures $\displaystyle \grtree{T}{a} := \res{\tree{T}}{\grset{T}{a}}$, $\displaystyle \greqtree{T}{a} := \res{\tree{T}}{\greqset{T}{a}}$, $\displaystyle \comptree{T}{a} := \res{\tree{T}}{\compset{T}{a}}$ and $\displaystyle \compeqtree{T}{a} := \res{\tree{T}}{\compeqset{T}{a}}$ will be called respectively the \defstyle{open cone} generated by $a$, the \defstyle{closed cone} generated by $a$, the \defstyle{closed anticone} generated by $a$, and the \defstyle{open anticone} generated by $a$.  These structures are depicted in Figures \ref{Fig:Cones} and \ref{Fig:Anticones}.

\begin{figure}
	\begin{center}
		\includegraphics[scale=0.8]{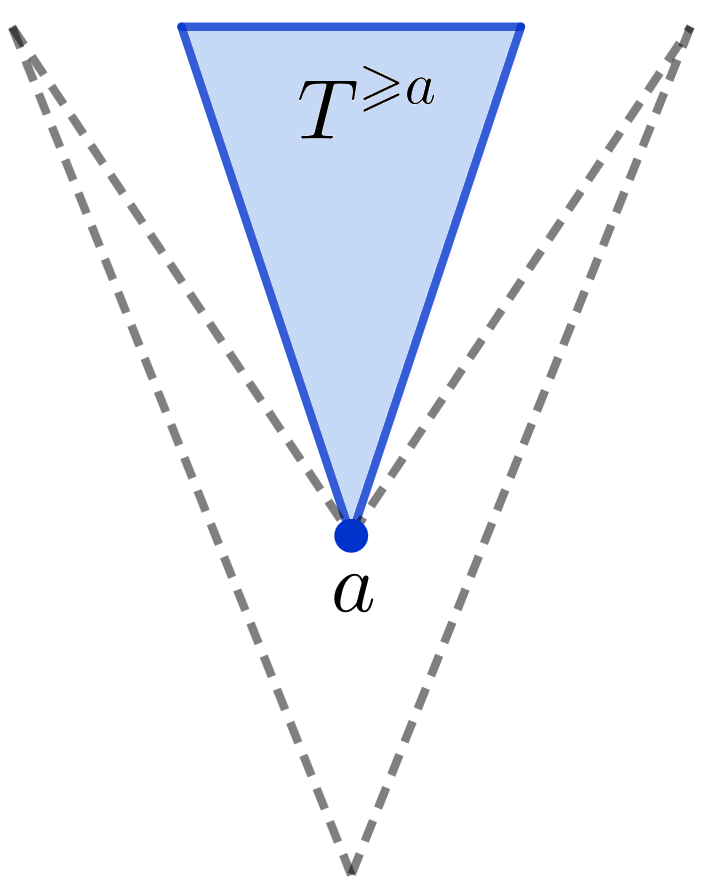} \hspace{5mm}
		\includegraphics[scale=0.8]{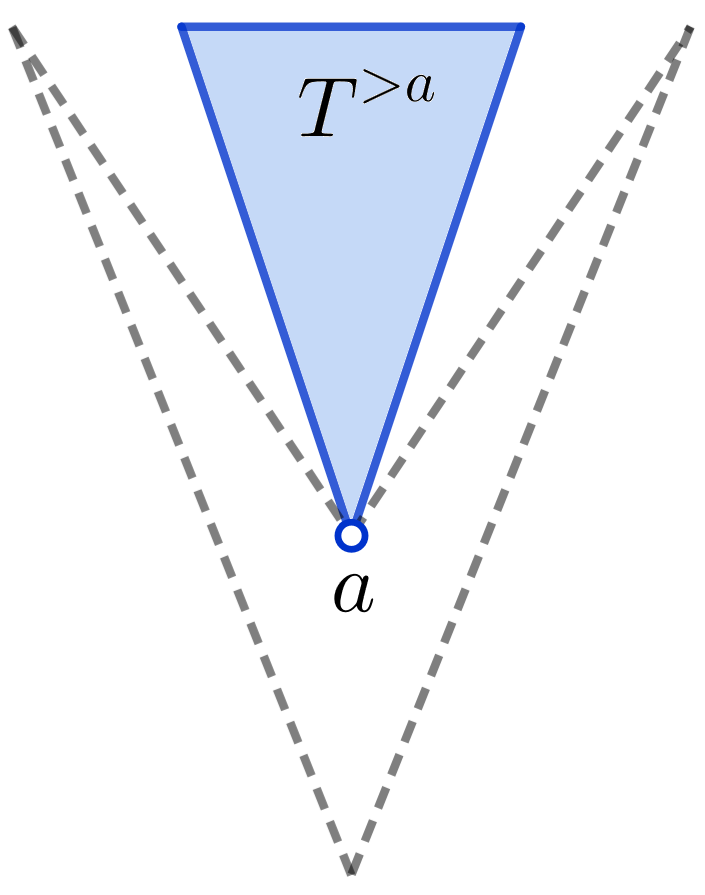}
		\caption{Cones generated by a node $a$ in a tree $\tree{T}$.}
		\label{Fig:Cones}
	\end{center}
\end{figure}

\begin{figure}
	\begin{center}
		\includegraphics[scale=0.8]{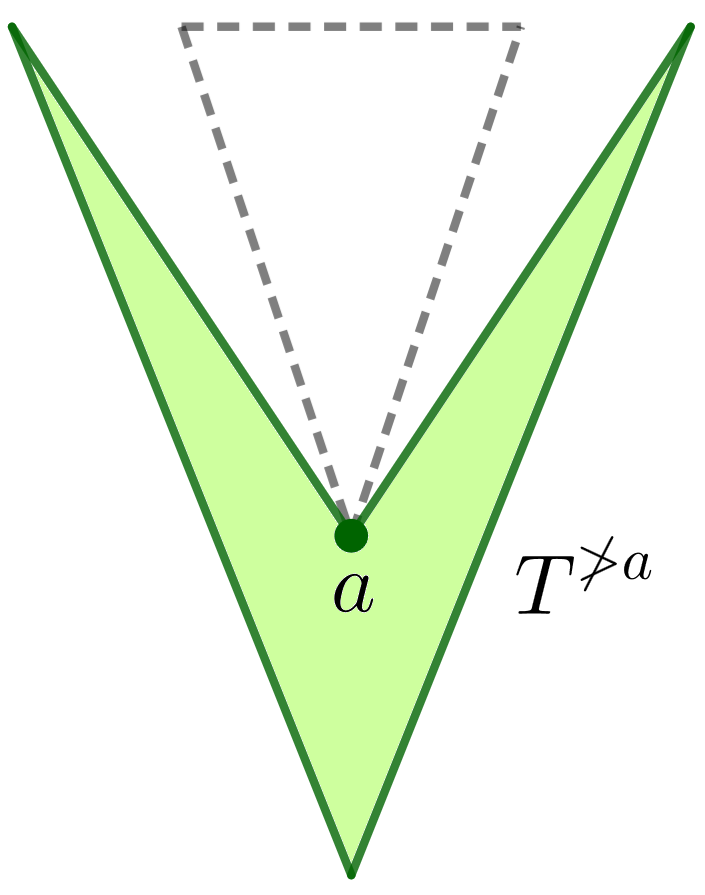} \hspace{5mm}
		\includegraphics[scale=0.8]{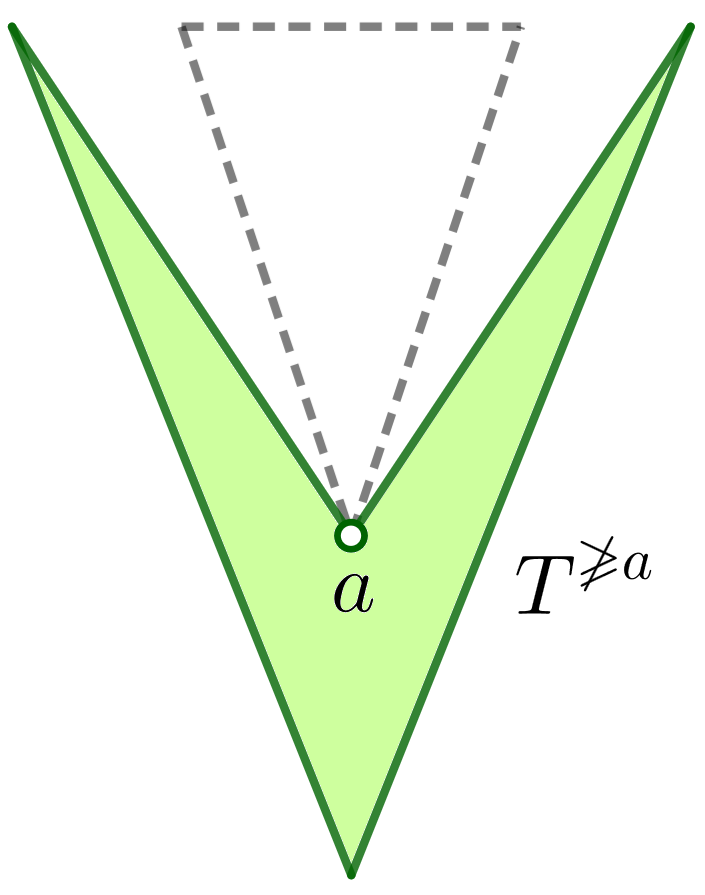}
		\caption{Anticones generated by a node $a$ in a tree $\tree{T}$.}
		\label{Fig:Anticones}
	\end{center}
\end{figure}

\begin{lemma} \label{Thm:ConeEquivalence}
	Let $\tree{T}$ and $\tree{S}$ be $\lang$-trees and let $a \in T$ and $b \in S$ such that $(\tree{T};a) \eq{n} (\tree{S};b)$.
	
	\medskip
	
	\noindent (i) \ \ $\grtree{T}{a} \eq{n} \grtree{S}{b}$ \hspace{5mm}
	(ii) \ \ $\left(\greqtree{T}{a};a\right) \eq{n} \left(\greqtree{S}{b};b\right)$ \hspace{5mm}
	(iii) \ \ $\left(\comptree{T}{a};a\right) \eq{n} \left(\comptree{S}{b};b\right)$

	\noindent (iv) $\left(\comptree{T}{a};(-\infty,a]\right) \eq{n} \left(\comptree{S}{b};(-\infty,b]\right)$ \hspace{5mm}
	(v) \ \ $\left(\compeqtree{T}{a};(-\infty,a)\right) \eq{n} \left(\compeqtree{S}{b};(-\infty,b)\right)$
\end{lemma}

\begin{proof}
	By the straightforward use of Ehrenfeucht-Fra\"iss\'e games.
\end{proof}

\subsection{Co-forests}

We next define co-forests of nodes, which will be used to partition well-founded trees in a natural way; these substructures are depicted in Figure \ref{Fig:Co-Forest}.  Given $a \in T$, define
$$\cfset{T}{a} := \left(\bigcap_{t \ls{T} a} \greqset{T}{t} \right) \setminus \grset{T}{a}.$$
The \defstyle{co-forest} of $a$ is the structure $\cftree{T}{a} := \res{\tree{T}}{\cfset{T}{a}}$.

\begin{figure}
	\begin{center}
		\includegraphics[scale=0.18]{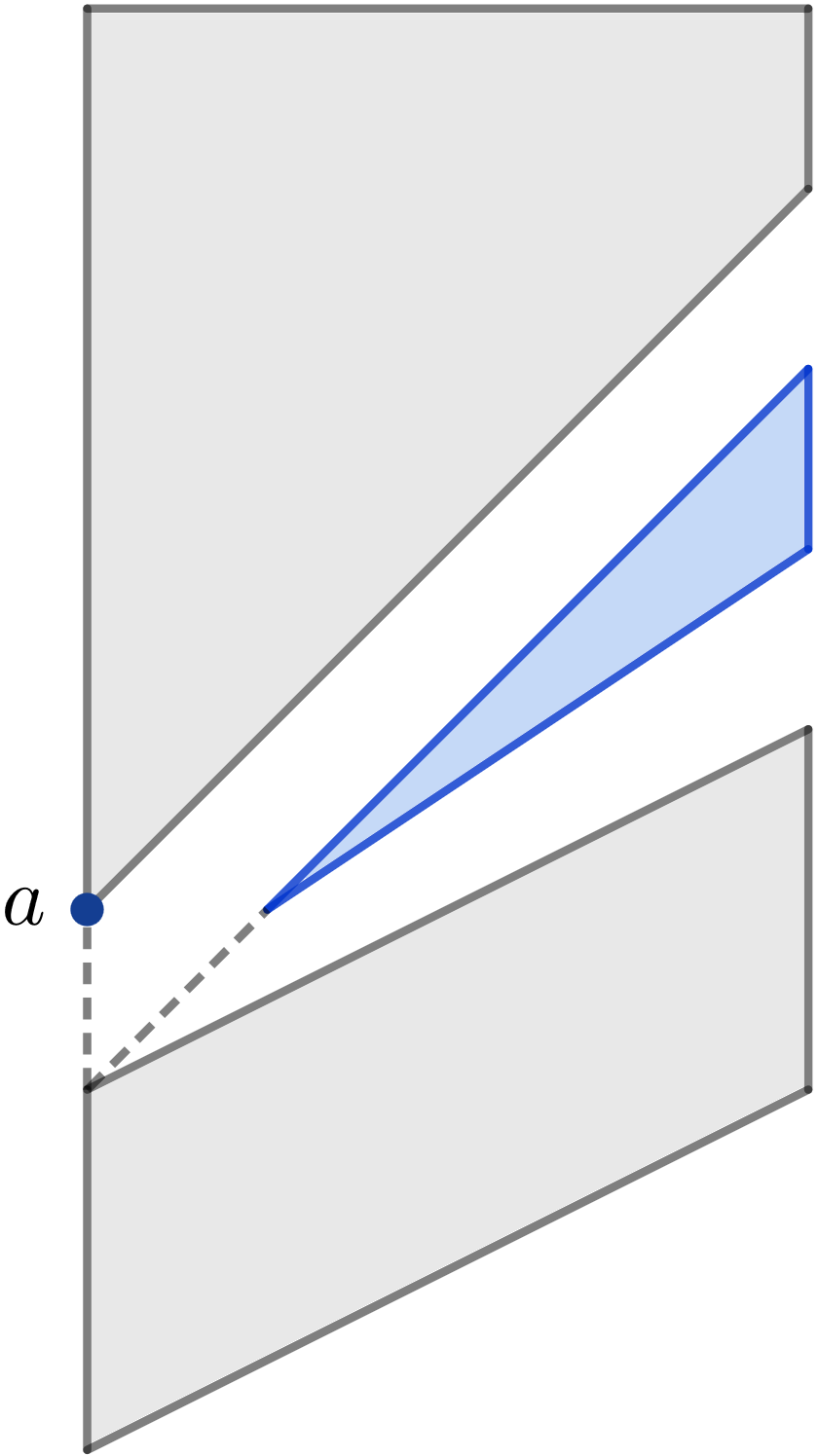}
		\caption{The set $\cfset{T}{a}$ is depicted in blue.}
		\label{Fig:Co-Forest}
	\end{center}
\end{figure}

\begin{proposition} \label{Thm:Partition}
	Let $\tree{T}$ be a well-founded tree and let $A$ be a path in $\tree{T}$.  For each $b \in T$ there exists $a \in A$ such that $b \in \cfset{T}{a}$.
\end{proposition}

\begin{proof}
	By the well-foundedness of $\tree{T}$, the set $\bigcup \{ \greqset{T}{t} : t \in A \ \text{with} \ t \not< b \}$ contains a minimal element $a$.  It is readily seen that $a \in A$ and $b \in \cfset{T}{a}$.
\end{proof}

Define
$$\for{cf}_{z}(x) := (z \not< x) \wedge \forall y \, \big( (y < z) \rightarrow (y < x) \big).$$
The set $\cfset{T}{a}$ can be defined in $(\tree{T};a)$ by the formula $\for{cf}_{a}(x)$.  Observe that $\for{cf}_{z}(x)$ has quantifier rank $1$.

\begin{lemma} \label{Thm:SideForestEquivalence}
	Let $n$ be a natural number, let $\tree{T}$ and $\tree{S}$ be $\lang$-trees, and let $a \in T$ and $b \in S$ such that $(\tree{T};a) \eq{n+1} (\tree{S};b)$.  Then $\left(\cftree{T}{a};a\right) \eq{n} \left(\cftree{S}{b};b\right)$.
\end{lemma}

\begin{proof}
	Since $(\tree{T};a) \eq{n+1} (\tree{S};b)$ and $\for{cf}_{z}(x)$ has quantifier rank $1$ then $\left(\cftree{T}{a};a\right) = (\tree{T};a)^{\for{cf}_{a}} \eq{n} (\tree{S};b)^{\for{cf}_{b}} = \left(\cftree{S}{b};b\right)$.
\end{proof}

\subsection{Cross sections}

Given nodes $a$ and $b$ in a tree $\tree{T}$ with $a<b$, let
\begin{eqnarray*}
	\opensecset{T}{a}{b} & := & \bigcup_{x \in (a,b)} \cfset{T}{x} \\
	\text{and} \ \ \ \rightclosedsecset{T}{a}{b} & := & \bigcup_{x \in (a,b]} \cfset{T}{x}.
\end{eqnarray*}
Hence $\opensecset{T}{a}{b} = \grset{T}{a} \setminus \left( \grset{T}{b} \cup \cfset{T}{b} \right)$ and $\rightclosedsecset{T}{a}{b} = \opensecset{T}{a}{b} \cup \cfset{T}{b} = \grset{T}{a} \setminus \grset{T}{b}$.  These sets are depicted in Figure \ref{Fig:CrossSection}.  The trees $\opensectree{T}{a}{b} := \res{\tree{T}}{\opensecset{T}{a}{b}}$ and $\rightclosedsectree{T}{a}{b} := \res{\tree{T}}{\rightclosedsecset{T}{a}{b}}$ will be called respectively the \defstyle{open cross section} and \defstyle{closed cross section} of $\tree{T}$ with respect to $a$ and $b$.

\begin{figure}
	\begin{center}
		\includegraphics[scale=0.3]{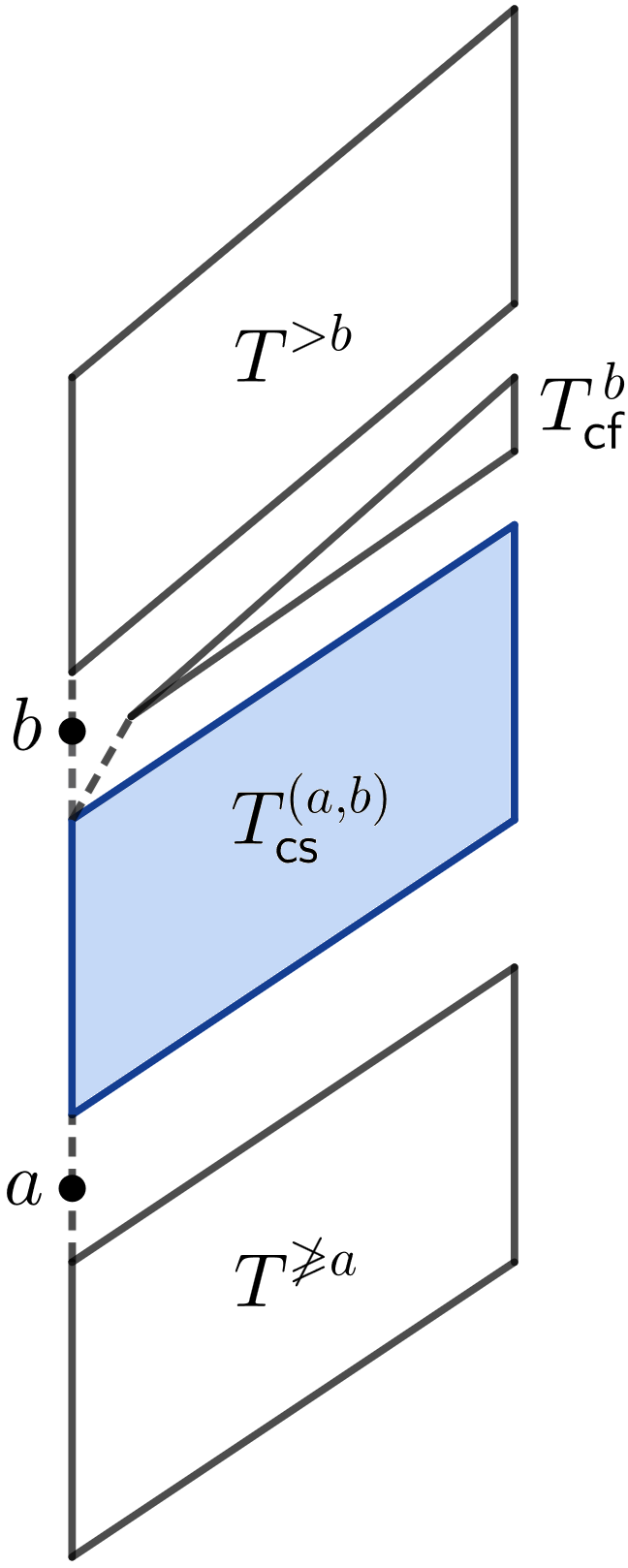}
		\hspace{2cm}
		\includegraphics[scale=0.3]{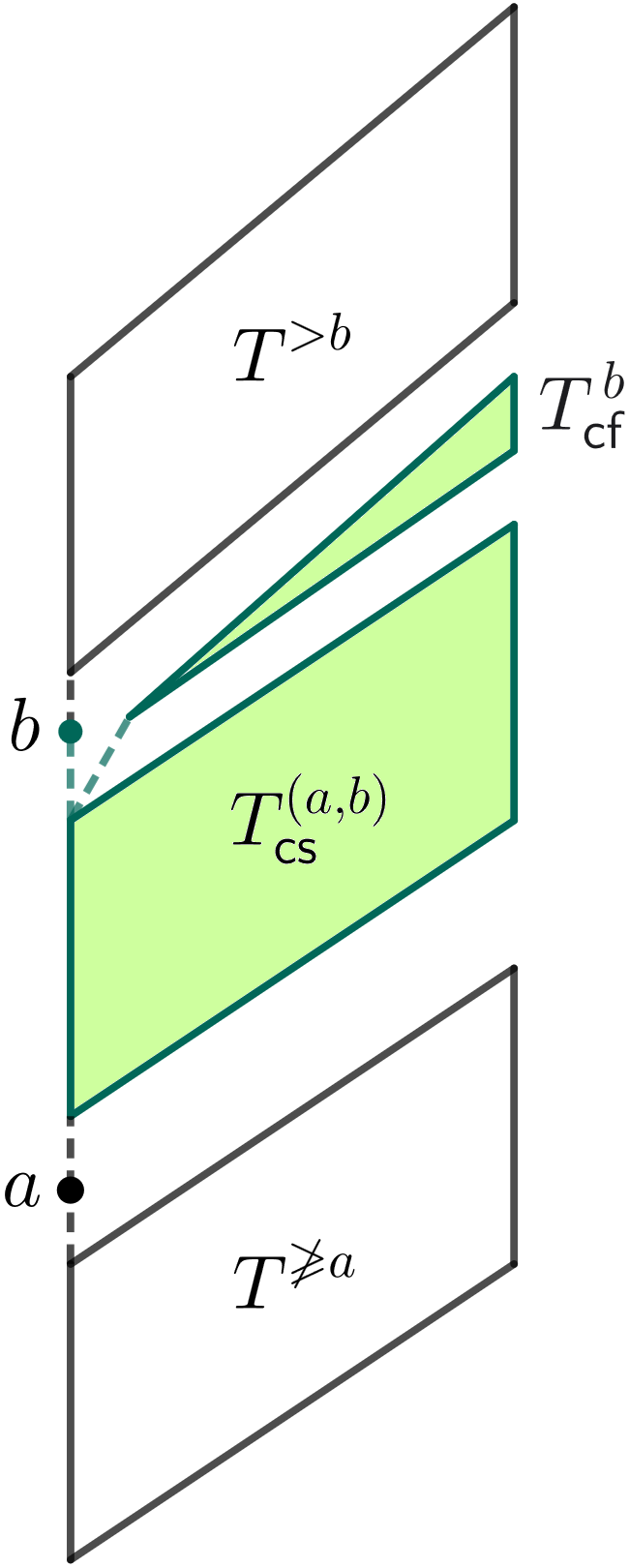}
		\caption{The set $\opensecset{T}{a}{b}$ is depicted in blue; the set $\rightclosedsecset{T}{a}{b}$ is depicted in green.}
		\label{Fig:CrossSection}
	\end{center}
\end{figure}

Letting
\begin{eqnarray*}
	\for{ocs}_{y,z}(x) & := & (x > y) \wedge \neg (x > z) \wedge \neg \for{cf}_{z}(x) \\
	\text{and} \ \ \ \for{ccs}_{y,z}(x) & := & \for{ocs}_{y,z}(x) \vee \for{cf}_{z}(x),
\end{eqnarray*}
observe that $\for{ocs}_{a,b}(x)$ defines the set $\opensecset{T}{a}{b}$ in $(\tree{T};a,b)$, and $\for{ccs}_{a,b}(x)$ defines the set $\rightclosedsecset{T}{a}{b}$ in $(\tree{T};a,b)$, with $\for{ocs}_{y,z}$ and $\for{ccs}_{y,z}$ each having quantifier rank $1$.

\begin{lemma} \label{Thm:CrossSectionEquivalence}
	Let $n$ be a natural number, let $\tree{T}$ and $\tree{S}$ be $\lang$-trees, and let $a,b \in T$ and $c,d \in S$ with $a \ls{T} b$ and $c \ls{S} d$ and such that $(\tree{T};a,b) \eq{n+1} (\tree{S};c,d)$.  Then $\left(\opensectree{T}{a}{b};(a,b)\right) \eq{n} \left(\opensectree{S}{c}{d};(c,d)\right)$.
\end{lemma}

\begin{proof}
	Observe that the predicate $(a,b)$ can be defined in $(\tree{T};a,b)$ by the quantifier-free formula $(a < x) \wedge (x < b)$, and $(c,d)$ can be defined in $(\tree{S};c,d)$ by the formula $(c < x) \wedge (x < d)$.  Hence since $(\tree{T};a,b) \eq{n+1} (\tree{S};c,d)$ then also $(\tree{T};a,b;(a,b)) \eq{n+1} (\tree{S};c,d;(c,d))$.  From the fact that $\for{ocs}_{y,z}(x)$ has quantifier rank $1$, we get
	$$\left(\opensectree{T}{a}{b};(a,b)\right) = (\tree{T};a,b;(a,b))^{\for{ocs}_{a,b}} \eq{n} (\tree{S};c,d;(c,d))^{\for{ocs}_{c,d}} = \left(\opensectree{S}{c}{d};(c,d)\right).$$
	
	\qedspace{24pt}
\end{proof}

\begin{lemma} \label{Thm:OpenCrossSectionEquivalence}
	Let $\tree{T}$ and $\tree{S}$ be $\lang$-trees and let $a,b \in T$ and $c,d \in S$ with $a \ls{T} b$ and $c \ls{S} d$ such that $(\tree{T};a) \eq{0} (\tree{S};c)$ and $(\tree{T};b) \eq{1} (\tree{S};d)$.  Then $\left(\opensectree{T}{a}{b};(a,b)\right) \eq{0} \left(\opensectree{S}{c}{d};(c,d)\right)$.
\end{lemma}

\begin{proof}
	From the fact that $a \ls{T} b$ and $c \ls{S} d$ and $(\tree{T};a) \eq{0} (\tree{S};c)$ and\linebreak$(\tree{T};b) \eq{1} (\tree{S};d)$, it follows that $(\tree{T};a,b) \eq{0} (\tree{S};c,d)$.  In order to conclude that\linebreak$\left(\opensectree{T}{a}{b};(a,b)\right) \eq{0} \left(\opensectree{S}{c}{d};(c,d)\right)$, we need therefore only establish that for each constant symbol $p$ in $\lang$,
	$$p^{\tree{T}} \in \opensecset{T}{a}{b} \ \Longleftrightarrow \ p^{\tree{S}} \in \opensecset{S}{c}{d}.$$
	For this, it suffices to show that for each constant symbol $p$ in $\lang$,
	$$p^{\tree{T}} \in \cfset{T}{b} \ \Longleftrightarrow \ p^{\tree{S}} \in \cfset{S}{d}.$$
	This in turn follows from the fact that $(\tree{T};b) \eq{1} (\tree{S};d)$ and co-forests can be defined by a formula of quantifier rank $1$.
\end{proof}

\begin{proposition} \label{Thm:ClosedSection}
	Let $\tree{T}$ and $\tree{S}$ be $\lang$-trees and let $a,b \in T$ and $c,d \in S$ with $a \ls{T} b$ and $c \ls{S} d$ and such that $\left(\opensectree{T}{a}{b};(a,b)\right) \eq{n} \left(\opensectree{S}{c}{d};(c,d)\right)$ and $(\tree{T};b) \eq{n+1} (\tree{S};d)$.  Then $\left(\rightclosedsectree{T}{a}{b};(a,b]\right) \eq{n} \left(\rightclosedsectree{S}{c}{d};(c,d]\right)$.
\end{proposition}

\begin{proof}
	By Lemma \ref{Thm:SideForestEquivalence}, $\left(\cftree{T}{b};b\right) \eq{n} \left(\cftree{S}{d};d\right)$.  By Lemma \ref{Thm:SumsOfTrees},
	\begin{multline} \nonumber
		\left(\rightclosedsectree{T}{a}{b};(a,b),b\right) \cong \left(\opensectree{T}{a}{b};(a,b)\right) +_{(a,b)} \left(\cftree{T}{b};b\right) \eq{n} \\ \left(\opensectree{S}{c}{d};(c,d)\right) +_{(c,d)} \left(\cftree{S}{d};d\right) \cong \left(\rightclosedsectree{S}{c}{d};(c,d),d\right).
	\end{multline}
	The predicate $(a,b]$ can be defined in the tree $\left(\rightclosedsectree{T}{a}{b};(a,b),b\right)$ by the formula $x \leqslant b$, and $(c,d]$ can be defined in $\left(\rightclosedsectree{S}{c}{d};(c,d),d\right)$ by the formula $x \leqslant d$.  It follows that $\left(\rightclosedsectree{T}{a}{b};(a,b),b,(a,b]\right) \eq{n} \left(\rightclosedsectree{S}{c}{d};(c,d),d,(c,d]\right)$ and hence $\left(\rightclosedsectree{T}{a}{b};(a,b]\right) \eq{n} \left(\rightclosedsectree{S}{c}{d};(c,d]\right)$.
\end{proof}

\section{Properties of trees}  \label{Sec:Properties}

\subsection{Properties involving leaves}

A tree $\tree{T}$ will be called \defstyle{ideal} when for each leaf $x \in T$, $x$ has no immediate predecessor, i.e.
$$\forall x \big(\for{leaf}(x) \rightarrow \forall y \big((y < x) \rightarrow \exists z (y < z < x)\big)\big).$$
$\tree{T}$ will be called \defstyle{monofolic} when for each leaf $x \in T$, $x$ is the only node for which $(-\infty,x) \ll x$, i.e.
$$\forall x \big( \for{leaf}(x) \rightarrow \forall y \big( \forall z ((z < x) \rightarrow (z < y)) \rightarrow (y = x) \big) \big).$$

\subsection{Variegated trees}

Let $A$ be a path in $\tree{T}$ and let $n$ be a natural number.  Define $\ell(A) := \emptyset$ if $A$ has no leaf, and $\ell(A) := \{a\}$ if $a$ is the leaf of $A$.  $A^-$ will denote the leafless part of $A$, i.e.\ $A^- = A \setminus \ell(A)$.  The \defstyle{$n$-spectrum} of $A$ in $\tree{T}$, denoted $\spec{n}{\tree{T}}{A}$, is the set of integers $i$ for which the set $\{ t \in A^- : \tree{T} \models \tau^n_i(t) \}$ is cofinal in $A^-$, where $\tau^n_1,\ldots,\tau^n_m$ are the $n$-characteristic formulas with one free variable in the language $\lang(\tree{T})$.  $\tree{T}$ is \defstyle{variegated} when for each natural number $n$, each open path $B$ in $\tree{T}$, and each tuple $\bar{a}$ from $T$, if $B$ has $n$-spectrum $X$ in $(\tree{T};\bar{a})$ then there exists a closed path $C$ in $\tree{T}$ such that $C$ also has $n$-spectrum $X$ in $(\tree{T};\bar{a})$; note that here the characteristic formulas $\tau^n_i$ are considered in the language $\lang(\tree{T};\bar{a})$.

\subsection{Blocks}

Let $\tree{T}$ be a tree.  For $X \subseteq T$, define
$$\block[\tree{T}]{n}{X} := \left\{ i : \tree{T} \models \tau^n_i(t) \ \text{for some} \ t \in X \right\},$$
where the characteristic formulas $\tau^n_i$ are considered in the language $\lang(\tree{T})$.  If the tree $\tree{T}$ being worked in is clear, $\block[\tree{T}]{n}{X}$ will be written simply as $\block{n}{X}$.  If $X$ is the interval $(x,y)$ then $\block[\tree{T}]{n}{X} = \block[\tree{T}]{n}{(x,y)}$ will be written simply as $\block[\tree{T}]{n}{x,y}$.  Given a set $M$ of positive integers, the set $X$ will be called an \defstyle{$(n,M)$-block} when $\block[\tree{T}]{n}{X} = M$.  Observe that for any tree $\tree{S}$ in the language $\lang(\tree{T})$ and $Y$ a set of nodes in $\tree{S}$,
\begin{equation} \label{Eqn:Block}
	\block[\tree{T}]{n+1}{X} = \block[\tree{S}]{n+1}{Y} \ \Longrightarrow \ \block[\tree{T}]{n}{X} = \block[\tree{S}]{n}{Y}.
\end{equation}
Let $\isablock{M}{n}(x_1,x_2)$ be the formula
\begin{equation} \label{Eqn:Beta}
	\forall y \left( (x_1 < y < x_2) \rightarrow \bigvee_{i \in M} \tau^n_i(y) \right) \wedge \left(\bigwedge_{i \in M} \exists y \left((x_1 < y < x_2) \wedge \tau^n_i(y)\right)\right),
\end{equation}
which has quantifier rank $n+1$ and carries the following meaning: for nodes $a_1,a_2 \in T$, $\tree{T} \models \isablock{M}{n}(a_1,a_2)$ if and only if $\block[\tree{T}]{n}{a_1,a_2} = M$.

\subsection{Balanced and focal trees}

Given a tree $\tree{T}$ and a node $a \in T$, define
$$\psubset{T}{a} := \{ x \in T : x \smile a \}$$
and $\psubtree{T}{a} := \res{\tree{T}}{\psubset{T}{a}}$.  The structure $\psubtree{T}{a}$ is a tree that will be called the \defstyle{induced subtree} of $\tree{T}$ generated by $a$.  $\tree{T}$ will be called \defstyle{balanced} when for each natural number $n$ and all nodes $x,y \in T$, $\block{n}{\psubset{T}{x}} = \block{n}{\psubset{T}{y}}$.  The property of being balanced can be expressed by the axiom scheme $\for{BAL}$ that contains the sentence
$$\forall x \forall y \big( \exists z \big( (z \smile x) \wedge \varphi(z) \big) \rightarrow \exists z \big( (z \smile y) \wedge \varphi(z) \big) \big)$$
for each $\lang(\tree{T})$-formula $\varphi(x)$.
$\tree{T}$ will be called \defstyle{hereditarily balanced} when $\greqtree{T}{t}$ is balanced for each node $t \in T$, i.e.\ when $\block[\greqtree{T}{t}]{n}{\greqset{T}{t}(x)} = \block[\greqtree{T}{t}]{n}{\greqset{T}{t}(y)}$ for all $x,y \in \greqset{T}{t}$.  The property of being hereditarily balanced can be expressed using the set of sentences $\left\{ \forall v \left(\sigma^{\theta}\right) : \sigma \in \for{BAL} \right\}$, where $\sigma^{\theta}$ denotes the relativisation of $\sigma$ to the formula $\theta(w,v) := (v \leqslant w)$.\footnote{We use the notation of \cite{Rosenstein} for indicating relativisations.  Thus to obtain $\sigma^{\theta}$ from $\sigma$, each subformula of $\sigma$ of the form $\forall x (\psi)$ is replaced by the formula $\forall x \left( (u \leqslant x) \rightarrow \psi^{\theta}\right)$, and each subformula of $\sigma$ of the form $\exists x (\psi)$ is replaced by the formula $\exists x \left( (u \leqslant x) \wedge \psi^{\theta}\right)$.}

$\tree{T}$ will be called \defstyle{focal} when, for each node $a \in T$ and each $\lang(\tree{T})$-formula $\varphi(x,y)$, if there exists a node $b \in \greqset{T}{a}$ such that $\tree{T} \models \varphi(a,b)$, then for each node $c \in \greqset{T}{a}$ there exists $d \in \greqset{T}{a}$ with $d \smile c$ such that $\tree{T} \models \varphi(a,d)$.  The property of focality can be expressed using the axiom scheme that contains the sentence
$$\forall x \forall y \big( \big( (x \leqslant y) \wedge \varphi(x,y) \big) \rightarrow \forall u \big( (x \leqslant u) \rightarrow \exists w \big( (x \leqslant w) \wedge (w \smile u) \wedge \varphi(x,w) \big) \big) \big)$$
for each $\lang(\tree{T})$-formula $\varphi(x,y)$.

\begin{proposition}
	Let $\tree{T}$ be a tree.  $\tree{T}$ is hereditarily balanced if and only if $\tree{T}$ is focal.
\end{proposition}

\begin{proof}
	$(\Longrightarrow)$: Suppose that $\tree{T}$ is hereditarily balanced.  Let $n$ be a natural number and let $\varphi(x,y)$ be an $\lang(\tree{T})$-formula of quantifier rank $n$.  Let $a \in T$ and $b,c \in \greqset{T}{a}$ be nodes such that $\tree{T} \models \varphi(a,b)$.  Let $i$ be that integer for which $\greqtree{T}{a} \models \tau^{n+1}_i(b)$, where the characteristic formulas $\tau^{n+1}_j$ are taken in the language $\lang(\greqtree{T}{a})$.  Since $\tree{T}$ is hereditarily balanced, $\block[\greqtree{T}{a}]{n+1}{\greqset{T}{a}(b)} = \block[\greqtree{T}{a}]{n+1}{\greqset{T}{a}(c)}$, hence there exists $d \in \greqset{T}{a}$ with $d \smile c$ such that $\greqtree{T}{a} \models \tau^{n+1}_i(d)$, from which $(\greqtree{T}{a};b) \eq{n+1} (\greqtree{T}{a};d)$, hence $(\greqtree{T}{a};a,b) \eq{n} (\greqtree{T}{a};a,d)$.  Since trivially $(\compeqtree{T}{a};(-\infty,a)) \eq{n} (\compeqtree{T}{a};(-\infty,a))$ then by Lemma \ref{Thm:SumsOfTrees},
	$$(\tree{T};a,b) \cong \compeqtree{T}{a} +_{(-\infty,a)} (\greqtree{T}{a};a,b) \eq{n} \compeqtree{T}{a} +_{(-\infty,a)} (\greqtree{T}{a};a,d) \cong (\tree{T};a,d)$$
	hence $\tree{T} \models \varphi(a,d)$, as required.
	
	$(\Longleftarrow)$: Suppose that $\tree{T}$ is focal.  Let $a \in T$ and let $b,c \in \greqset{T}{a}$.  It must be shown that $\block[\greqtree{T}{a}]{n}{\greqset{T}{a}(b)} = \block[\greqtree{T}{a}]{n}{\greqset{T}{a}(c)}$.  Let $u \in \greqset{T}{a}(b)$ and let $i$ be that integer for which $\greqtree{T}{a} \models \tau^n_i(u)$.  Observe that $\tree{T} \models \cfor{T}{a,u}{n}(a,u)$.  By the focality of $\tree{T}$ there exists $v \in \greqset{T}{a}(c)$ such that $\tree{T} \models \cfor{T}{a,u}{n}(a,v)$, hence $(\tree{T};a,u) \eq{n} (\tree{T};a,v)$.  By Lemma \ref{Thm:ConeEquivalence}, $(\tree{T};u)^{\geqslant a} \eq{n} (\tree{T};v)^{\geqslant a}$, from which $(\greqtree{T}{a};u) \eq{n} (\greqtree{T}{a};v)$, hence $\greqtree{T}{a} \models \tau^n_i(v)$, as required.
\end{proof}

\section{Equivalence of cross sections}  \label{Sec:Equivalence}

The following result adapts \cite[Lemma 3.2]{MwesigyeTruss} to the setting of trees.

\begin{lemma} \label{Thm:IntervalTheorem}
	Let $n$ be a natural number.  Suppose that
	\begin{enumerate}[(I)]
		\item
			$\tree{T}$ and $\tree{S}$ are focal well-founded $\lang$-trees;
		\item
			$a_1,a_2 \in T$ such that $a_1 \ls{T} a_2$, and $b_1,b_2 \in S$ such that $b_1 \ls{S} b_2$;
		\item
			$\left(\tree{T};a_1\right) \eq{n+2} \left(\tree{S};b_1\right)$ and $\left(\tree{T};a_2\right) \eq{n+2} \left(\tree{S};b_2\right)$;
		\item
			for some set $M$ of positive integers,
			\begin{enumerate}[(i)]
				\item
					$(a_1,a_2)$ can be partitioned into $2^n-1$ subintervals, each of which is an \break $(n+1,M)$-block, and
				\item
					$(b_1,b_2)$ can be partitioned into $2^n-1$ subintervals, each of which is an \break $(n+1,M)$-block.
			\end{enumerate}
	\end{enumerate}
	Then $\left(\opensectree{T}{a_1}{a_2};(a_1,a_2)\right) \eq{n} \left(\opensectree{S}{b_1}{b_2};(b_1,b_2)\right)$.
\end{lemma}

\begin{proof}
	Let $P(n,\tree{T},\tree{S},a_1,a_2,b_1,b_2)$ denote the conjunction of properties $(I)$ to $(IV)$.  We prove the result using induction on $n$.  It follows from Lemma \ref{Thm:OpenCrossSectionEquivalence} that $\left(\opensectree{T}{a_1}{a_2};(a_1,a_2)\right) \eq{0} \left(\opensectree{S}{b_1}{b_2};(b_1,b_2)\right)$ whenever $P(0,\tree{T},\tree{S},a_1,a_2,b_1,b_2)$.
	
	Fix a natural number $k$ for which $\left(\opensectree{T}{a_1}{a_2};(a_1,a_2)\right) \eq{k} \left(\opensectree{S}{b_1}{b_2};(b_1,b_2)\right)$ whenever $P(k,\tree{T},\tree{S},a_1,a_2,b_1,b_2)$, and assume $P(k+1,\tree{U},\tree{V},c_1,c_2,d_1,d_2)$.  Then
	\begin{eqnarray}
		\left(\tree{U};c_1\right) & \eq{k+3} & \left(\tree{V};d_1\right) \label{Eqn:IntervalTheorem1} \\
		\text{and} \ \ \ \left(\tree{U};c_2\right) & \eq{k+3} & \left(\tree{V};d_2\right). \label{Eqn:IntervalTheorem2}
	\end{eqnarray}
	We now use Fact \ref{Thm:BackAndForth} to show that $\left(\opensectree{U}{c_1}{c_2};(c_1,c_2)\right) \eq{k+1} \left(\opensectree{V}{d_1}{d_2};(d_1,d_2)\right)$.  To this end, let $p \in \opensecset{U}{c_1}{c_2}$.  We will exhibit $q \in \opensecset{V}{d_1}{d_2}$ such that $\left(\opensectree{U}{c_1}{c_2};(c_1,c_2),p\right) \eq{k} \left(\opensectree{V}{d_1}{d_2};(d_1,d_2),q\right)$.  If instead we are given $p \in \opensecset{V}{d_1}{d_2}$, the argument for finding $q \in \opensecset{U}{c_1}{c_2}$ such that $\left(\opensectree{U}{c_1}{c_2};(c_1,c_2),q\right) \eq{k} \left(\opensectree{V}{d_1}{d_2};(d_1,d_2),p\right)$, is similar.
	
	By assumption, there exists a set $N$ of positive integers such that $(c_1,c_2)$ can be partitioned into $2^{k+1}-1$ intervals, each of which is a $(k+2,N)$-block, and $(d_1,d_2)$ can be partitioned into $2^{k+1}-1$ intervals, each of which is a $(k+2,N)$-block.  From \eqref{Eqn:Block} it follows that there exists a set $Q$ of positive integers such that each of the aforementioned subintervals of $(c_1,c_2)$ is a $(k+1,Q)$-block, and each of the aforementioned subintervals of $(d_1,d_2)$ is a $(k+1,Q)$-block.  Using the fact that $2^{k+1}-1 = (2^k-1)+1+(2^k-1)$, it hence follows that $(c_1,c_2)$ can be partitioned into intervals $I_1$, $I_2$ and $I_3$ with $I_1 \double[T] I_2 \double[T] I_3$, and $(d_1,d_2)$ can be partitioned into intervals $J_1$, $J_2$ and $J_3$ with $J_1 \double[S] J_2 \double[S] J_3$, such that
	\begin{enumerate}[(A)]
		\item
			$I_1$ and $I_3$ can each be partitioned into $2^k-1$ subintervals, each of which is both a $(k+2,N)$-block and a $(k+1,Q)$-block, and
		\item
			$J_1$ and $J_3$ can each be partitioned into $2^k-1$ subintervals, each of which is both a $(k+2,N)$-block and a $(k+1,Q)$-block, and
		\item
			$I_2$ is both a $(k+2,N)$-block and a $(k+1,Q)$-block, and $J_2$ is both a $(k+2,N)$-block and a $(k+1,Q)$-block.
	\end{enumerate}
	Let $p_0$ be that node on $(c_1,c_2)$ for which $p \in \cfset{U}{p_0}$; such $p_0$ exists by Proposition \ref{Thm:Partition}.  We will consider the following cases: (Case 1) $p_0 \in I_1 \cup I_2$, and (Case 2) $p_0 \in I_3$.
	
	Case 1: $p_0 \in I_1 \cup I_2$.  By \eqref{Eqn:IntervalTheorem1} there exists $r_0 \gr{T} d_1$ such that
	\begin{equation} \label{Eqn:IntervalTheoremCase1No1}
		\left(\tree{U};c_1,p_0\right) \eq{k+2} \left(\tree{V};d_1,r_0\right).
	\end{equation}
	By the focality of $\tree{V}$, there exists $r_1 \comp{T} d_2$ such that
	\begin{equation} \label{Eqn:IntervalTheoremCase1No2}
		\left(\tree{U};c_1,p_0\right) \eq{k+2} \left(\tree{V};d_1,r_1\right)
	\end{equation}
	hence, by Lemma \ref{Thm:CrossSectionEquivalence},
	\begin{equation} \label{Eqn:IntervalTheoremCase1No3}
		\left(\opensectree{U}{c_1}{p_0};(c_1,p_0)\right) \eq{k+1} \left(\opensectree{V}{d_1}{r_1};(d_1,r_1)\right).
	\end{equation}
	
	We next identify a node $q_0 \in J_1 \cup J_2$ for which it will be shown that
	\begin{eqnarray}
		\left(\opensectree{U}{c_1}{p_0};(c_1,p_0)\right) & \eq{k\phantom{+1}} & \left(\opensectree{V}{d_1}{q_0};(d_1,q_0)\right), \label{Eqn:IntervalTheoremCase1No4} \\
		\left(\cftree{U}{p_0};p_0\right) & \eq{k+1} & \left(\cftree{V}{q_0};q_0\right) \label{Eqn:IntervalTheoremCase1No5} \\
		\text{and} \ \ \left(\opensectree{U}{p_0}{c_2};(p_0,c_2)\right) & \eq{k\phantom{+1}} & \left(\opensectree{V}{q_0}{d_2};(q_0,d_2)\right). \label{Eqn:IntervalTheoremCase1No6}
	\end{eqnarray}
	For this, we distinguish two subcases: (Case 1.1) $r_1 \in J_1 \cup J_2$ and (Case 1.2) $r_1 \not\in J_1 \cup J_2$.
	
	Case 1.1: $r_1 \in J_1 \cup J_2$.  Simply take $q_0 := r_1$.  Then \eqref{Eqn:IntervalTheoremCase1No4} follows immediately from \eqref{Eqn:IntervalTheoremCase1No3}.  From \eqref{Eqn:IntervalTheoremCase1No2} we obtain
	\begin{equation} \label{Eqn:IntervalTheoremCase1No7}
		\left(\tree{U};p_0\right) \eq{k+2} \left(\tree{V};q_0\right),
	\end{equation}
	upon which \eqref{Eqn:IntervalTheoremCase1No5} follows by an application of Lemma \ref{Thm:SideForestEquivalence}.  Finally, it follows from (A), (B) and (C) that $(p_0,c_2)$ can be partitioned into $2^k-1$ subintervals, each of which is a $(k+1,Q)$-block, and $(q_0,d_2)$ can be partitioned into $2^k-1$ subintervals, each of which is a $(k+1,Q)$-block.  Along with \eqref{Eqn:IntervalTheorem2} and \eqref{Eqn:IntervalTheoremCase1No7} this gives $P(k,\tree{U},\tree{V},p_0,c_2,q_0,d_2)$, from which \eqref{Eqn:IntervalTheoremCase1No6} is obtained using the inductive hypothesis.
	
	Case 1.2: $r_1 \not\in J_1 \cup J_2$.  From the fact that $\block{k+2}{I_1 \cup I_2} = \block{k+2}{J_2} = N$, along with \eqref{Eqn:IntervalTheoremCase1No2}, it follows that there exists $q_0 \in J_2$ such that
	\begin{equation} \label{Eqn:IntervalTheoremCase1No8}
		\left(\tree{U};p_0\right) \eq{k+2} \left(\tree{V};r_1\right) \eq{k+2} \left(\tree{V};q_0\right).
	\end{equation}
	Suppose that $\block{k+1}{c_1,p_0} = R$ and note that $R \subseteq Q$.  Then $\tree{U} \models \isablock{R}{k+1}(c_1,p_0)$, hence from \eqref{Eqn:IntervalTheoremCase1No2} we get $\tree{V} \models \isablock{R}{k+1}(d_1,r_1)$, hence $\block{k+1}{d_1,r_1} = R$.  However, since $r_1 \not\in J_1 \cup J_2$ then $Q \subseteq R$, hence $R = Q$, hence $\block{k+1}{d_1,r_1} = R = Q = \block{k+1}{d_1,q_0}$.  It follows that $(d_1,r_1)$ and $(d_1,q_0)$ can each be partitioned into $2^k-1$ subintervals, each of which is a $(k+1,Q)$-block.  From this, along with \eqref{Eqn:IntervalTheoremCase1No8}, it follows that $P(k,\tree{V},\tree{V},d_1,r_1,d_1,q_0)$, so by the inductive hypothesis, $\left(\opensectree{V}{d_1}{r_1};(d_1,r_1)\right) \eq{k} \left(\opensectree{V}{d_1}{q_0};(d_1,q_0)\right)$.  Property \eqref{Eqn:IntervalTheoremCase1No4} now follows using \eqref{Eqn:IntervalTheoremCase1No3}.
	
	Property \eqref{Eqn:IntervalTheoremCase1No5} follows from \eqref{Eqn:IntervalTheoremCase1No8} again by an application of Lemma \ref{Thm:SideForestEquivalence}.
	
	Finally, observe as in Case 1.1 that $P(k,\tree{U},\tree{V},p_0,c_2,q_0,d_2)$, from which \eqref{Eqn:IntervalTheoremCase1No6} is again obtained using the inductive hypothesis.  This concludes Case 1.2.
	
	By \eqref{Eqn:IntervalTheoremCase1No5} there exists $q \in \cfset{V}{q_0}$ such that
	\begin{equation} \label{Eqn:IntervalTheoremCase1No9}
		\left(\cftree{U}{p_0};p_0,p\right) \eq{k} \left(\cftree{V}{q_0};q_0,q\right).
	\end{equation}
	
	By \eqref{Eqn:IntervalTheoremCase1No4}, \eqref{Eqn:IntervalTheoremCase1No9} and Lemma \ref{Thm:SumsOfTrees},
	\begin{multline}
		\left(\rightclosedsectree{U}{c_1}{p_0};(c_1,p_0),p_0,p\right) \ \cong \ \opensectree{U}{c_1}{p_0} +_{(c_1,p_0)} \left(\cftree{U}{p_0};p_0,p\right) \ \eq{k} \\
		\opensectree{V}{d_1}{q_0} +_{(d_1,q_0)} \left(\cftree{V}{q_0};q_0,q\right) \ \cong \ \left(\rightclosedsectree{V}{d_1}{q_0};(d_1,q_0),q_0,q\right). \label{Eqn:IntervalTheoremCase1No10}
	\end{multline}
	
	Treating $(c_1,p_0)$ and $(d_1,q_0)$ as unary predicates, note that $(c_1,p_0]$ can be defined in $\left(\rightclosedsectree{U}{c_1}{p_0};(c_1,p_0),p_0,p\right)$ by the quantifier-free formula $(x \in (c_1,p_0)) \vee (x = p_0)$, and $(d_1,q_0]$ can be defined in $\left(\rightclosedsectree{V}{d_1}{q_0};(d_1,q_0),q_0,q\right)$ by the quantifier-free formula $(x \in (d_1,q_0)) \vee (x = q_0)$.  It hence follows from \eqref{Eqn:IntervalTheoremCase1No10} that
	\begin{equation} \label{Eqn:IntervalTheoremCase1No11}
		\left(\rightclosedsectree{U}{c_1}{p_0};(c_1,p_0],p\right) \eq{k} \left(\rightclosedsectree{V}{d_1}{q_0};(d_1,q_0],q\right).
	\end{equation}
	
	By \eqref{Eqn:IntervalTheoremCase1No6}, \eqref{Eqn:IntervalTheoremCase1No11} and Lemma \ref{Thm:SumsOfTrees},
	
	\begin{multline}
		\left(\opensectree{U}{c_1}{c_2};(c_1,p_0],p,(p_0,c_2)\right) \ \cong \ \left(\rightclosedsectree{U}{c_1}{p_0};(c_1,p_0],p\right) +_{(c_1,p_0]} \left(\opensectree{U}{p_0}{c_2};(p_0,c_2)\right) \ \eq{k} \\
		\left(\rightclosedsectree{V}{d_1}{q_0};(d_1,q_0],q\right) +_{(d_1,q_0]} \left(\opensectree{V}{q_0}{d_2};(q_0,d_2)\right) \ \cong \ \left(\opensectree{V}{d_1}{d_2};(d_1,q_0],q,(q_0,d_2)\right). \label{Eqn:IntervalTheoremCase1No12}
	\end{multline}
	
	Note that the predicate $(c_1,c_2)$ can be defined in $\left(\opensectree{U}{c_1}{c_2};(c_1,p_0],p,(p_0,c_2)\right)$ by the formula $(x \in (c_1,p_0]) \vee (x \in (p_0,c_2))$, and the predicate $(d_1,d_2)$ can be defined in $\left(\opensectree{V}{d_1}{d_2};(d_1,q_0],q,(q_0,d_2)\right)$ by the formula $(x \in (d_1,q_0]) \vee (x \in (q_0,d_2))$, hence it follows from \eqref{Eqn:IntervalTheoremCase1No12} that
	$$\left(\opensectree{U}{c_1}{c_2};(c_1,c_2),p\right) \ \eq{k} \ \left(\opensectree{V}{d_1}{d_2};(d_1,d_2),q\right),$$
	as required.
	
	 Case 2: $p_0 \in I_3$.  The argument is quite similar to the one used in Case 1; we will outline it, emphasizing how it differs from the argument used in Case 1.  This time there must exist $r_1 \in (-\infty,d_2)$ for which \eqref{Eqn:IntervalTheoremCase1No2} becomes $\left(\tree{U};c_2,p_0\right) \eq{k+2} \left(\tree{V};d_2,r_1\right)$ (focality is not needed here), hence \eqref{Eqn:IntervalTheoremCase1No3} becomes $\left(\opensectree{U}{p_0}{c_2};(p_0,c_2)\right) \eq{k+1} \left(\opensectree{V}{r_1}{d_2};(r_1,d_2)\right)$.  We then exhibit $q_0 \in J_2 \cup J_3$ satisfying \eqref{Eqn:IntervalTheoremCase1No4}, \eqref{Eqn:IntervalTheoremCase1No5} and \eqref{Eqn:IntervalTheoremCase1No6} by considering the following two subcases: (Case 2.1) $r_1 \in J_2 \cup J_3$ and (Case 2.2) $r_1 \not\in J_2 \cup J_3$.
	
	In Case 2.1, again take $q_0$ to be $r_1$.  Arguments similar to those used in Case 1.1 can be used to show that \eqref{Eqn:IntervalTheoremCase1No4}, \eqref{Eqn:IntervalTheoremCase1No5} and \eqref{Eqn:IntervalTheoremCase1No6} hold.
	
	In Case 2.2, use the fact that $\block{k+2}{I_3} = \block{k+2}{J_2} = N$ in order to find $q_0 \in J_2$ that satisfies \eqref{Eqn:IntervalTheoremCase1No8}.  Then $P(k,\tree{U},\tree{V},c_1,p_0,d_1,q_0)$, and the inductive hypothesis can be used to establish \eqref{Eqn:IntervalTheoremCase1No4}.  Property \eqref{Eqn:IntervalTheoremCase1No5} again follows from \eqref{Eqn:IntervalTheoremCase1No8} using Lemma \ref{Thm:SideForestEquivalence}.  To obtain \eqref{Eqn:IntervalTheoremCase1No6}, assume that $\block{k+1}{p_0,c_2} = R$ in order to deduce that $\block{k+1}{r_1,d_2} = R$.  Since $r_1 \not\in J_2 \cup J_3$, it again follows that $R = Q$, hence $\block{k+1}{r_1,d_2} = R = Q = \block{k+1}{q_0,d_2}$, so the intervals $(r_1,d_2)$ and $(q_0,d_2)$ can each be partitioned into $2^k-1$ subintervals, each of which is a $(k+1,Q)$-block.  From this we get $P(k,\tree{V},\tree{V},r_1,d_2,q_0,d_2)$ so by the inductive hypothesis, $\left(\opensectree{V}{r_1}{d_2};(r_1,d_2)\right) \eq{k} \left(\opensectree{V}{q_0}{d_2};(q_0,d_2)\right)$, hence also $\left(\opensectree{U}{p_0}{c_2};(p_0,c_2)\right) \eq{k} \left(\opensectree{V}{q_0}{d_2};(q_0,d_2)\right)$.  This establishes \eqref{Eqn:IntervalTheoremCase1No6}.
	
	Again there exists a node $q \in \cfset{V}{q_0}$ that satisfies \eqref{Eqn:IntervalTheoremCase1No9}, from which \eqref{Eqn:IntervalTheoremCase1No11} and \eqref{Eqn:IntervalTheoremCase1No12} can again be shown to hold.  As in Case 1, it then follows that\break $\left(\opensectree{U}{c_1}{c_2};(c_1,c_2),p\right) \ \eq{k} \ \left(\opensectree{V}{d_1}{d_2};(d_1,d_2),q\right)$.
\end{proof}

\begin{corollary} \label{Thm:IntervalCorollary}
	Let $n$ be a natural number and suppose that $P(n+1,\tree{T},\tree{S},a_1,a_2,b_1,b_2)$ from the proof of Lemma \ref{Thm:IntervalTheorem} holds.  Then $\left(\rightclosedsectree{T}{a_1}{a_2};(a_1,a_2]\right) \eq{n} \left(\rightclosedsectree{S}{b_1}{b_2};(b_1,b_2]\right)$.
\end{corollary}

\begin{proof}
	By Lemma \ref{Thm:IntervalTheorem} and Proposition \ref{Thm:ClosedSection}.
\end{proof}

\section{Adding leaves to unbounded paths}  \label{Sec:Leaves}

\begin{proposition} \label{Thm:EquivalentPaths}
	Let $n$ be a natural number.  Let $\tree{T}$ and $\tree{S}$ be $\lang$-trees that are ideal, monofolic, well-founded and focal.  If $A$ is a path in $\tree{T}$, and $B$ is a path in $\tree{S}$, such that $A$ and $B$ have the same $(n+3)$-spectra in $\tree{T}$ and $\tree{S}$ respectively, then $(\tres{T}{\ell(A)}; A^-) \eq{n} (\tres{S}{\ell(B)}; B^-)$.
\end{proposition}

\begin{proof}
	Observe that the ($n+2$)-spectrum of $A$ in $\tree{T}$ equals the ($n+2$)-spectrum of $B$ in $\tree{S}$; let this spectrum be $C$.  Since $A$ and $B$ have the same $(n+3)$-spectra in $\tree{T}$ and $\tree{S}$ respectively, there exists a sequence $(t_k)_{k \in \alpha}$ that is cofinal in $A^-$, and a sequence $(s_k)_{k \in \beta}$ that is cofinal in $B^-$, such that
	\begin{enumerate}[(i)]
		\item
			$(\tree{T};t_i) \eq{n+3} (\tree{S};s_j)$ for all $i \in \alpha$ and $j \in \beta$, and
		\item
			for all $i,j \in \alpha$ and $k,m \in \beta$ with $i < j$ and $k < m$, each of the intervals $(t_i,t_j)$ and $(s_k,s_m)$ can be partitioned into $2^n-1$ subintervals, with each of these subintervals being an $(n+2,C)$-block.
	\end{enumerate}
	By Corollary \ref{Thm:IntervalCorollary}, $\left(\rightclosedsectree{T}{t_i}{t_j};(t_i,t_j]\right) \eq{n} \left(\rightclosedsectree{S}{s_k}{s_m};(s_k,s_m]\right)$ for all $i,j \in \alpha$ and $k,m \in \beta$ with $i < j$ and $k < m$, and by Lemma \ref{Thm:ConeEquivalence}, $\left(\comptree{T}{t_0};(-\infty,t_0]\right) \eq{n} \left(\comptree{S}{s_0};(-\infty,s_0]\right)$.  We will show that $(\tres{T}{\ell(A)}; A^-) \eq{n} (\tres{S}{\ell(B)}; B^-)$ by describing a winning strategy for Player II for the game $\ef{n}{(\tres{T}{\ell(A)}; A^-)}{(\tres{S}{\ell(B)}; B^-)}$.
	
	\textbf{Round 1:}  Set $x_0 := t_0$ and $y_0 := s_0$.  Suppose that Player I chooses the node $a_1 \in T \setminus \ell(A)$ for his first move (the case in which he instead chooses $b_1 \in S \setminus \ell(B)$ is similar).
	\begin{itemize}
		\item
			Case 1: $a_1 \in \compset{T}{t_0}$.  Then Player II uses her winning strategy for the game $\ef{n}{\left(\comptree{T}{t_0};(-\infty,t_0]\right)}{\left(\comptree{S}{s_0};(-\infty,s_0]\right)}$ to choose a node $b_1 \in \compset{S}{s_0}$ as her response, and we set $x_1 := t_0$ and $y_1 := s_0$.
		\item
			Case 2: $a_1 \not\in \compset{T}{t_0}$.  Let $p \in \alpha$ be the smallest value for which $a_1 \in \compset{T}{t_p}$.  Player II then chooses a node $b_1 \in \rightclosedsecset{S}{s_0}{s_1}$ as her response using her winning strategy for the game $\ef{n}{\left(\rightclosedsectree{T}{t_0}{t_p};(t_0,t_p]\right)}{\left(\rightclosedsectree{S}{s_0}{s_1};(s_0,s_1]\right)}$.  Set $x_1 := t_p$ and $y_1 := s_1$.
	\end{itemize}
	
	\textbf{Round $k+1$:}  Now let $1 \leqslant k < n$ and suppose that the nodes $a_1,\ldots,a_k \in T \setminus \ell(A)$ and $b_1,\ldots,b_k \in S \setminus \ell(B)$ were chosen in the first $k$ rounds of the game, and that $X := \{x_0,\ldots,x_k\}$ and $Y := \{y_0,\ldots,y_k\}$ have been defined such that for all $i$ and $j$,
	\begin{enumerate}[(i)]
		\item
			$x_i \in \{t_m : m \in \alpha\}$ and $y_i \in \{s_m : m \in \beta\}$, and
		\item
			$x_i \ls{T} x_j$ if and only if $y_i \ls{S} y_j$, and
		\item
			$a_i \in \compset{T}{x_i}$ and $b_i \in \compset{T}{y_i}$, while
		\item
			$a_i \not\in \compset{T}{x}$ whenever $x \in X$ with $x \ls{T} x_i$, and $b_i \not\in \compset{S}{y}$ whenever $y \in Y$ with $y \ls{S} y_i$.
	\end{enumerate}
	Suppose that for his ($k+1$)-th move, Player I chooses the node $a_{k+1} \in T \setminus \ell(A)$ (again the case in which he instead chooses $b_{k+1} \in S \setminus \ell(B)$ is similar).  We consider three cases:
	\begin{itemize}
		\item
			Case 1: $a_{k+1} \in \compset{T}{t_0}$.  Player II responds with $b_{k+1} \in \compset{S}{s_0}$ using her winning strategy for the game $\ef{n}{\left(\comptree{T}{t_0};(-\infty,t_0]\right)}{\left(\comptree{S}{s_0};(-\infty,s_0]\right)}$ and based upon the nodes that have already been played in the sets $\compset{T}{t_0}$ and $\compset{S}{s_0}$.  Let $x_{k+1} := t_0$ and $y_{k+1} := s_0$.
		\item
			Case 2: $a_{k+1} \in \rightclosedsecset{T}{x}{x'}$ for some $x,x' \in X$ with $x < x'$.  Let $l,u \in \{0,\ldots,k\}$ be any integers such that $x_l$ is the largest node in $X$, and $x_u$ is the smallest node in $X$, for which $a_{k+1} \in \rightclosedsecset{T}{x_l}{x_u}$.  Then Player II responds with $b_{k+1} \in \rightclosedsecset{S}{y_l}{y_u}$ using her winning strategy for the game $\ef{n}{\left(\rightclosedsectree{T}{x_l}{x_u};(x_l,x_u]\right)}{\left(\rightclosedsectree{S}{y_l}{y_u};(y_l,y_u]\right)}$ and based upon the nodes that have already been played in the sets $\rightclosedsecset{T}{x_l}{x_u}$ and $\rightclosedsecset{S}{y_l}{y_u}$.  Let $x_{k+1} := x_u$ and $y_{k+1} := y_u$.
		\item
			Case 3: $a_{k+1} > x$ for all $x \in X$.  Let $i \in \{0,\ldots,k\}$ be any integer such that $x_i$ is the largest node in $X$.  Suppose that $x_i$ is the node $t_j$, and suppose that $y_i$ is the node $s_m$.  Let $q \in \alpha$ be the smallest value for which $a_{k+1} \in \compset{T}{t_q}$.  Player II then chooses a node $b_{k+1} \in \rightclosedsecset{S}{s_m}{s_{m+1}}$ as her response using her winning strategy for the game $\ef{n}{\left(\rightclosedsectree{T}{t_j}{t_q};(t_j,t_q]\right)}{\left(\rightclosedsectree{S}{s_m}{s_{m+1}};(s_m,s_{m+1}]\right)}$ and based upon the nodes that have already been played in the sets $\rightclosedsecset{T}{t_j}{t_q}$ and $\rightclosedsecset{S}{s_m}{s_{m+1}}$.  Let $x_{k+1} := t_q$ and $y_{k+1} := s_{m+1}$.
	\end{itemize}
	It then follows that $\{(a_1,b_1),\ldots,(a_n,b_n)\}$ forms a local isomorphism between the trees $(\tres{T}{\ell(A)}; A^-)$ and $(\tres{S}{\ell(B)}; B^-)$, as required.
\end{proof}

\begin{proposition} \label{Thm:AddSingleLeaf}
	Let $\tree{T}$ be a tree that is ideal, monofolic, well-founded, focal and variegated.  Let $A$ be a leafless path in $\tree{T}$ and let $\tree{E} := (\{e\};\emptyset)$ be the trivial tree.  Then $\tree{T} \preceq \sumtree{T}{E}{A}$ and $\sumtree{T}{E}{A}$ is variegated.
\end{proposition}

\begin{proof}
	Assume without loss of generality that none of the nodes in $T$ are labelled $e$.  We first use the Tarski-Vaught Test to show that $\tree{T} \preceq \sumtree{T}{E}{A}$.  To this end, let $\bar{a}$ be a tuple of nodes in $T$ and let $\varphi(x,\bar{z})$ be a formula of quantifier rank $n$ for which $(\sumtree{T}{E}{A};\bar{a}) \models \varphi(e,\bar{a})$.  Since $\tree{T}$ is variegated, there exists a leaf $b \in T$ such that the path $(-\infty,b]$ has the same ($n+4$)-spectrum as the path $A$ in $(\tree{T};\bar{a})$.  By Proposition \ref{Thm:EquivalentPaths}, $(\tree{T}; \bar{a}; A) \eq{n+1} (\tres{T}{b}; \bar{a}; (-\infty,b))$.  Let $\tree{B} := (\{b\};\emptyset)$ be another copy of the trivial tree.  By Lemma \ref{Thm:SumsOfTrees},
	\begin{multline} \nonumber
		(\sumtree{T}{E}{A}; \bar{a}; A) \cong ((\tree{T};\bar{a}) +_A \tree{E}; A) \eq{n+1} \left((\tres{T}{b};\bar{a}) +_{(-\infty,b)} \tree{B}; (-\infty,b)\right) \cong \\ \left((\tres{T}{b}) +_{(-\infty,b)} \tree{B}; \bar{a}; (-\infty,b)\right) \cong (\tree{T}; \bar{a}; (-\infty,b)).
	\end{multline}
	It now follows that $(\sumtree{T}{E}{A};\bar{a},e) \eq{n} (\tree{T};\bar{a},b)$ hence $(\tree{T};\bar{a}) \models \varphi(b,\bar{a})$, as required.

	To see that $\sumtree{T}{E}{A}$ is variegated, let $C$ be an open path in $\sumtree{T}{E}{A}$ and let $\bar{c}$ be a tuple of nodes from $T \cup \{e\}$.  Choose $d \in C$ such that $e$, and all the nodes from $\bar{c}$, belong to $\compeqset{T}{d}$.  Let $X$ be the $k$-spectrum of $C$ in $(\tree{T};d)$.  Since $\tree{T}$ is variegated, there exists a leaf $f \in T$ such that the path $(-\infty,f]$ also has $k$-spectrum $X$ in $(\tree{T};d)$.  Observe that for all $t \in \greqset{T}{d} \cap C$ and $s \in \greqset{T}{d} \cap (-\infty,f)$, if $(\tree{T};d,t) \eq{k} (\tree{T};d,s)$ then by Lemma \ref{Thm:ConeEquivalence},
	$$\left(\left( \sumtree{T}{E}{A} \right)^{\geqslant d}; t\right) \cong (\greqtree{T}{d};t) \eq{k} (\greqtree{T}{d};s) \cong \left(\left( \sumtree{T}{E}{A} \right)^{\geqslant d}; s\right),$$
	and trivially $\left(\left(\sumtree{T}{E}{A}\right)^{\not\geqslant d}; \bar{c}; (-\infty,d)\right) \eq{k} \left(\left(\sumtree{T}{E}{A}\right)^{\not\geqslant d}; \bar{c}; (-\infty,d)\right)$, so by Lemma \ref{Thm:SumsOfTrees},
	\begin{multline} \nonumber
		(\sumtree{T}{E}{A};\bar{c},t) \cong \left(\left(\sumtree{T}{E}{A}\right)^{\not\geqslant d}; \bar{c}\right) +_{(-\infty,d)} \left(\left( \sumtree{T}{E}{A} \right)^{\geqslant d}; t\right) \ \eq{k} \\ \left(\left(\sumtree{T}{E}{A}\right)^{\not\geqslant d}; \bar{c}\right) +_{(-\infty,d)} \left(\left( \sumtree{T}{E}{A} \right)^{\geqslant d}; s\right) \cong (\sumtree{T}{E}{A};\bar{c},s).
	\end{multline}
	It follows that the paths $C$ and $(-\infty,f]$ have the same $k$-spectrum in $(\sumtree{T}{E}{A};\bar{c})$, as required.
\end{proof}

\begin{theorem} \label{Thm: Embedding}
	Let $\tree{T}$ be a tree that is ideal, monofolic, well-founded, focal and variegated.  Well-order the set of leafless paths in $\tree{T}$ as $\mathcal{P} := \{ A_i : i \in \alpha \}$ for some limit ordinal $\alpha$.  Let $\tree{E}$ be the trivial tree.  For each $i \in \alpha$, define $\tree{T}_i$ as follows:
	\begin{itemize}
		\item
			$\tree{T}_0 := \tree{T}$;
		\item
			for each $j \in \alpha$, $\tree{T}_{j+1} := \tree{T}_j +_{A_j} \tree{E}$;
		\item
			for each limit ordinal $\gamma \in \alpha$, $\tree{T}_{\gamma} := \bigcup_{j \in \gamma} \tree{T}_j$.
	\end{itemize}
	Then
	\begin{enumerate}[(i)]
		\item
			$\tree{T}_i \preceq \tree{T}_j$ for all $i,j \in \alpha$ with $i < j$, and
		\item
			$\tree{T}_i$ is variegated for each $i \in \alpha$.
	\end{enumerate}
	Moreover, $\tree{T} \preceq \sumtree{T}{E}{\mathcal{P}}$.
\end{theorem}

\begin{proof}
	We prove the result using induction.  Let $\beta \in \alpha$ and suppose that (i) and (ii) hold for all $i$ and $j$ with $i,j < \beta$.  It suffices to show that $\tree{T}_i \preceq \tree{T}_{\beta}$ for each $i < \beta$, and $\tree{T}_{\beta}$ is variegated.

	Using Proposition \ref{Thm:AddSingleLeaf} and the Tarski-Vaught theorem on unions of elementary chains, it follows immediately that $\tree{T}_i \preceq \tree{T}_{\beta}$ for each $i < \beta$.

	To see that $\tree{T}_{\beta}$ is variegated, let $A$ be an open path in $\tree{T}_{\beta}$, let $\bar{a}$ be a tuple of nodes in $T_{\beta}$, and let $X$ be the $k$-spectrum of $A$ in $(\tree{T}_{\beta};\bar{a})$.  If $\beta$ is a successor ordinal, say $\beta = \gamma + 1$, use the fact that $\tree{T}_{\gamma}$ is variegated and apply Proposition \ref{Thm:AddSingleLeaf} to conclude that $\tree{T}_{\beta}$ is variegated.  If, on the other hand, $\beta$ is a limit ordinal, let $\gamma < \beta$ be such that all the nodes in $\bar{a}$ belong to $T_{\gamma}$.  Since $\tree{T}_{\gamma} \preceq \tree{T}_{\beta}$ then for each $t \in A$ it holds that $(\tree{T}_{\gamma};\bar{a},t) \equiv (\tree{T}_{\beta};\bar{a},t)$, so the $k$-spectrum of $A$ in $(\tree{T}_{\gamma};\bar{a})$ is also $X$.  Since $\tree{T}_{\gamma}$ is variegated, there is a leaf $b \in T_{\gamma}$ such that the $k$-spectrum of $(-\infty,b]$ in $(\tree{T}_{\gamma};\bar{a})$ is $X$.  Observe that $b$ is also a leaf in $\tree{T}_{\beta}$.  Again using the fact that $\tree{T}_{\gamma} \preceq \tree{T}_{\beta}$, the $k$-spectrum of $(-\infty,b]$ in $(\tree{T}_{\beta};\bar{a})$ is also $X$, as required.

	Finally, it follows from (i) and the Tarski-Vaught theorem on unions of elementary chains that $\tree{T} \preceq \bigcup_{i \in \alpha} \tree{T}_i \cong \sumtree{T}{E}{\mathcal{P}}$.
\end{proof}

\begin{comment}
	Note that the tree $\sumtree{T}{E}{\mathcal{P}}$ in Theorem \ref{Thm: Embedding} is bounded.  Hence every tree that is ideal, monofolic, well-founded, focal and variegated, can be elementarily embedded in a bounded tree.
\end{comment}

\bibliographystyle{alpha}
\bibliography{trees.bib}{}

\end{document}